\documentclass[12pt,a4paper]{article}
\usepackage[utf8]{inputenc}
\usepackage{amsmath, amssymb, amsthm, verbatim, hyperref}
\usepackage{mathrsfs}
\usepackage[dvipsnames]{xcolor}
\usepackage{ragged2e}
\usepackage{marginnote}
\usepackage{enumerate}
\usepackage{makecell}
\usepackage{longtable}
\usepackage{epsfig}
\usepackage{ytableau}
\usepackage{tikz}
\usepackage{ytableau}
\usepackage{hyperref}
\usepackage{caption}
\usepackage{subcaption}
\hypersetup{hidelinks}
\usepackage{fancyvrb}
\usepackage{tikz-cd}
\usepackage[left=2.50cm, right=2.50cm, top=2.00cm, bottom=2.00cm]{geometry}
\usetikzlibrary {arrows.meta}
\usepackage{ifthen}

\newtheorem{theorem}{Theorem}[section]
\newtheorem{lemma}[theorem]{Lemma}
\newtheorem{proposition}[theorem]{Proposition}
\newtheorem{corollary}{Corollary}[theorem]

\newcommand{\todo}[1]{{\color{green}{#1}}}

\def\fill{\operatorname{fill}}

\def\ins{\operatorname{ins}}
\def\rec{\operatorname{rec}}

\definecolor{goodgreen}{rgb}{0.01, 0.75, 0.24}

\theoremstyle{plain}

\newtheorem{thm}{Theorem}[section]

\newtheorem{cor}[thm]{Corollary}

\usepackage{ytableau}

\theoremstyle{definition}
\newtheorem{definition}[thm]{Definition}
\newtheorem{example}[thm]{Example}
\theoremstyle{remark}

\usepackage{float}

\newcommand{\SSS}{\mathfrak{S}}
\newcommand{\GGG}{\mathfrak{G}}

\newcommand{\BM}{\mathsf{BM}}

\def\+{\includegraphics[scale=0.4]{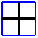}}
\def\bt{\includegraphics[scale=0.4]{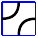}}

\usepackage{graphicx} 

\usepackage{amsmath}
\usepackage{amsthm}

\title{\vspace{-1em} 
\textbf{Pipe Dream Rectification and Dual RSK Correspondence}}
\author{AnLan Xu}
\date{\vspace{-2em}}
\allowdisplaybreaks

\begin{document}
\ytableausetup{centertableaux}

\maketitle

    \begin{abstract}
        We prove Dennin's conjecture that his variant of dual RSK correspondence is symmetric when restricted to biGrassmannian permutations. For a binary matrix $A$, let $A^\dagger$ denote its transpose-complement, and let $\operatorname{ins}(A)$ and $\operatorname{rec}(A)$ denote its insertion and recording tableaux. We prove that $\operatorname{ins}(A^\dagger) = \overline{\operatorname{rec}(A)}$, where the bar denotes the natural complement of the recording tableau. Our proof uses rectification of super pipe dreams and a downward induction on suffixes of $A$.
    \end{abstract}

\tableofcontents


\section{Introduction}

    The \textit{Robinson–Schensted–Knuth (RSK) correspondence} is a classic bijection in algebraic combinatorics, introduced by Knuth in \cite{knuth1970permutation} generalizing the RS correspondence \cite{robinson1938on,schensted1961longest}. It maps matrices of nonnegative integers to pairs of \textit{semistandard Young tableaux (SSYT)} of the same shape. For the many applications of this correspondence, we refer the reader to \cite{stanley2024enumerative}. Among the applications, it famously gives a simple combinatorial proof for the \textit{Cauchy identity}
    \[
    \sum_\lambda s_\lambda(\mathbf{x})s_\lambda(\mathbf{y}) = \prod_{i,j\geq 1} \dfrac{1}{1- x_iy_j}.
    \]
    Here, $s_{\lambda}(\mathbf{x})$ denotes the Schur function $s_{\lambda}$ in infinitely many variables $x_1,x_2,\ldots$. Dually, one also has the \textit{dual RSK correspondence} which bijects binary matrices (matrices with 0 and 1 entries) and pairs of SSYTs of conjugate shape. This yields the \textit{dual Cauchy identity}
    \[
    \sum_\lambda s_\lambda(\mathbf{x}) s_{\lambda'}(\mathbf{y}) = \prod_{i,j\geq 1} (1 + x_iy_j).
    \]
    These simple combinatorial proofs have fascinated combinatorialists for decades. Variants of the Cauchy identity and RSK have been found. See for example \cite{sagan1990robinson,fomingrothendieck,corwin2014tropical,bufetov2018hall,aigner2022qrst,frieden2024qt,dennin2025cauchy}.

    For each permutation $w\in S_n$, Lascoux and Sch\"utzenberger defined the \textit{Schubert polynomial} $\SSS_w(x)$ in \cite{lascoux1982polynomes} to represent the class of a corresponding Schubert variety $X_w$ inside the cohomology ring of the complete flag variety $\mathrm{Fl}_n$. In the equivariant cohomology setting, the classes are represented by \textit{double Schubert polynomials} $\SSS_w(x;y)$ \cite{lascoux1982classes}. Similarly, in the $K$-theory and equivariant $K$-theory settings, we have Grothedieck polynomials $\GGG_w(x)$ and double Grothendieck polynomials $\GGG_w(x;y)$, respectively \cite{lascoux1982structure,lascoux1983symmetry}. All these polynomials can be computed via \textit{pipe dreams} \cite{bergeron1993rc,fomin1996the,knutson2005grobner}.

    In \cite{fomingrothendieck}, Fomin--Kirillov introduced Grothendieck polynomials an additional parameter $\GGG_w^{(\beta)}(x)$ that specializes to $\SSS_w(x)$ and $\GGG_w(x)$ at $\beta = 0$ and $\beta = -1$, respectively. Then, the following identity can be considered a unifying Cauchy identity for Schubert and Grothendieck polynomials.

    \begin{thm}\label{thm:fomin-cauchy}
        We have
        \[
        \GGG_w^{(\beta)}(x;y) = \sum_{\substack{u,v\in S_\infty \\ w = u^{-1}\ast v}}\beta^{\ell(u)+ \ell(v) - \ell(w)}\GGG_v^{(\beta)}(x) \GGG_u^{(\beta)}(y),
        \]
        where $\ast$ denotes the product in the Demazure algebra.
    \end{thm}

    A combinatorial proof of Theorem \ref{thm:fomin-cauchy} was given by Dennin in \cite{dennin2025cauchy} via \textit{pipe dream rectification}, which sends a \textit{super pipe dream} to a pair of pipe dreams. See Section \todo{ref} for the relevant definitions.

    Recall also that for Grassmannian permutations, their Schubert polynomials are Schur polynomials. Thus, there is a bijection between reduced pipe dreams for Grassmannian permutations and SSYTs. Furthermore, there is a bijection between binary matrices and super pipe dreams for biGrassmannian permutations. Thus, Dennin observed that when restricted to biGrassmannian permutations, pipe dream rectification yields a variant of dual RSK correspondence
    \[
    \operatorname{dRSK}':A \mapsto (\operatorname{ins}(A),\operatorname{ins}(A^{\dagger})),
    \]
    for $A\in\BM_{m\times n}$, the set of $m\times n$ binary matrices. In addition, Dennin conjecture an additional symmetry property, which implies that this variant of dual RSK is indeed dual RSK. Our main theorem is that this conjecture is true.

    \begin{thm}[{\cite[Conjecture 8.9]{dennin2025cauchy}}]\label{thm:main-thm}
        For $A\in\BM_{m\times n}$,
        \[
        \overline{\operatorname{rec}(A)} = \operatorname{ins}(A^\dagger).
        \]
    \end{thm}

\section*{Acknowledgments}

I would first like to thank my mentor, Son Nguyen, for his guidance and support. His feedback and insight were invaluable throughout this project.

This work was carried out as part of MIT UROP. In particular, I would like to thank the Emma and Joseph Cherian UROP Fund for funding and providing the opportunity to pursue this research.

Finally, I would like to thank my friends and family for their support, encouragement, and willingness to listen to me talk through ideas during the project.

\section{Background}

\subsection{Pipe dreams and super pipe dreams}

    A \textit{pipe dream} of size $n$ is a collection of tiles on positions $(i,j)$, for $1\leq i,j\leq n$, such that each tile is either a cross \+ or a bump \bt. Here, $i$ labels the rows from top to bottom and $j$ labels the columns from left to right. A pipe dream is called \textit{reduced} if no two pipes cross more than once.

    The permutation of a pipe dream is obtained as follows. Label the pipes $i + j -1$. Then, read the pipes first left to right until the row ends then move to the next row from top to bottom. If two pipes have already crossed, then their next crossing is not counted. 

    \begin{example}
        Suppose our permutation is $w = 2143$. We can see that this permutation is made from $s_1s_3$. Because $w \in S_4$, our pipe dream is bounded by $1 \leq i, j \leq 4$. This gives us the three following reduced pipe dreams. 

        \begin{figure}[htbp]
        \centering
        \begin{subfigure}[b]{0.3\textwidth}
            \centering
            \includegraphics[scale = 0.35]{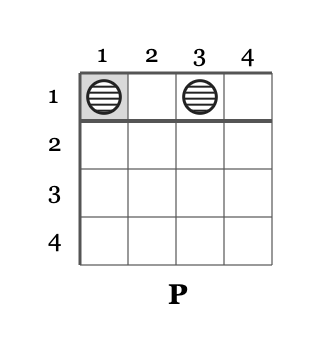}
        \end{subfigure}
        \hfill
        \begin{subfigure}[b]{0.3\textwidth}
            \centering
            \includegraphics[scale = 0.35]{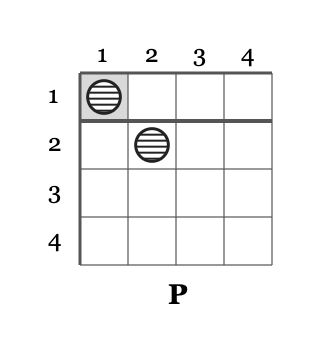}
        \end{subfigure}
        \hfill
        \begin{subfigure}[b]{0.3\textwidth}
            \centering
            \includegraphics[scale = 0.35]{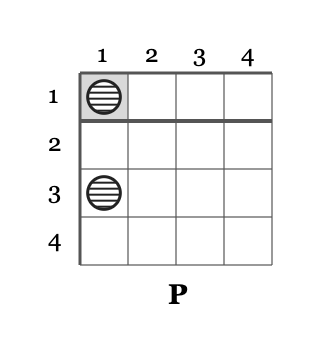}
        \end{subfigure}
        \caption{Reduced Pipe Dreams for Permutation $w=2143$}
        \label{fig:three_images}
    \end{figure}
    \end{example}

    The weight of a pipe dream $P$ is $\operatorname{wt}(P):=\prod_{p \text{ \+-tile}} x_{\operatorname{row}(p)}$. In addition, let $|P|$ denote the number of crossings in $P$. For a permutation \(w\), its \(\beta\)-Grothendieck polynomial is
    
    \[
        G_w^{(\beta)}(x)
        =
        \sum_{P\in PD^+(w)}
        \beta^{|P|-\ell(w)}
        \operatorname{wt}(P),
    \]
    where \(PD^+(w)\) is the set of not-necessarily-reduced pipe dreams with permutation \(w\), and
    \(\ell(w)\) is the length of \(w\).
    
    Observe that the contribution of a pipe dream comes only from the \+-tiles. Thus, it is enough to record the positions of the \+-tiles only.
    
    A \textit{super pipe dream} $\mathbf P=(P_x,P_y)$ consists of two colored collections of checkers. We represent the checkers in \(P_x\) by black, horizontally striped checkers and those in \(P_y\) by red, vertically striped checkers.
    
    The two colors encode the two components of the super pipe dream and allow the pipe-dream operations introduced below to distinguish the horizontal and vertical behavior of the corresponding pipes.
    
    \begin{example}
    An example of a super pipe dream is shown in Figure~\ref{fig:spd}.
    
    \begin{figure}[H]
        \centering
        \includegraphics[width=0.5\linewidth]{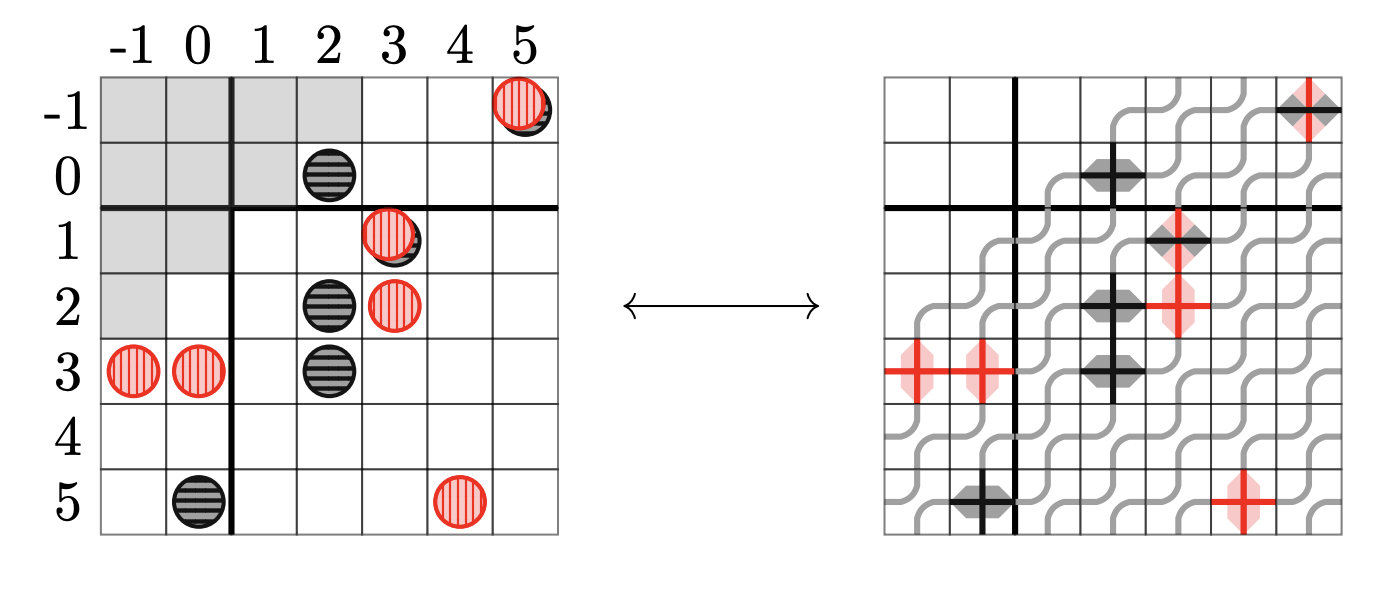}
        \caption{
            A super pipe dream
        }
        \label{fig:spd}
    \end{figure}
    \end{example}
    
    The two colors are related by transposition: locally, the behavior of a red checker is the transpose of the corresponding behavior of a black checker. This symmetry motivates the following operations.
    
    For a super pipe dream
    \[
        \mathbf P=(P_x,P_y),
    \]
    define
    \begin{align*}
        \text{Complement:}\qquad
        \overline{\mathbf P}
        &:= (P_y,P_x),\\
        \text{Transpose:}\qquad
        \mathbf P^{\mathsf t}
        &:= (P_x^{\mathsf t},P_y^{\mathsf t}),\\
        \text{Adjoint:}\qquad
        \mathbf P^\dagger
        &:= (P_y^{\mathsf t},P_x^{\mathsf t}).
    \end{align*}
    In particular,
    \[
        \mathbf P^\dagger
        =
        \overline{\mathbf P}^{\,\mathsf t}
        =
        \overline{\mathbf P^{\mathsf t}}.
    \]

    Let
    \[
        A=(A_{ij})\in BM_{m\times n}
    \]
    be a binary matrix; that is,
    \[
        A_{ij}\in\{0,1\}
    \]
    for every \(i\in[m]\) and \(j\in[n]\). Thus \(A\) has \(m\) rows and \(n\) columns.
    
    We associate a super pipe dream to \(A\) by placing
    \[
        \begin{cases}
            \text{a black checker at }(i,j), & A_{ij}=1,\\
            \text{a red checker at }(i,j),   & A_{ij}=0.
        \end{cases}
    \]
    Hence, every tile of the \(m\times n\) rectangle contains exactly one checker, and every binary matrix determines a corresponding super pipe dream.
    
    \begin{example}
        Let
        \[
            A=
            \begin{bmatrix}
                1 & 1 & 0\\
                0 & 1 & 0\\
                1 & 0 & 1
            \end{bmatrix}.
        \]
        Replacing every \(1\) by a black checker and every \(0\) by a red checker gives the super pipe dream shown in Figure~\ref{fig:bSPD}.
        
        \begin{figure}[H]
            \centering
            \includegraphics[width=0.2\linewidth]{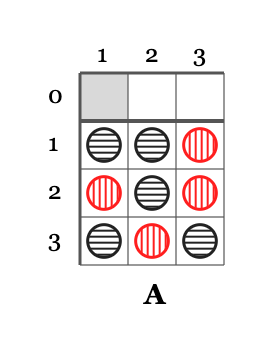}
            \caption{The super pipe dream associated with the binary matrix \(A\).}
            \label{fig:bSPD}
        \end{figure}
    \end{example}

\subsection{Rectification}

We review the flow operators defined in \cite[Section~4]{dennin2025cauchy}. Let \(\mathbf P\) be a super pipe dream with no red checkers in
column \(i+1\). First, we define the operator \(Y_i^+\), which maps \(\mathbf P\) to a super pipe dream $\textbf{Q}$ by flowing all red checkers in column \(i\) into column \(i+1\).
In particular, \(\mathbf Q\) has no red checkers in column
\(i\), and the operation preserves the number of black checkers in
each row.

Explicitly, \(Y_i^+\) is obtained by repeatedly choosing the lowest
red checker in column \(i\). If there is a black checker immediately to
its right, the move is the \(1\)-ladder as shown in Figure \ref{fig:Y^+ e.g. 1}.
Otherwise, one performs the larger
ladder move: the chosen red checker moves to the northeast corner of the
first available ladder, and the intermediate checkers are shifted
horizontally as shown in Figure \ref{fig:Y^+ e.g. 2}. In either case black checkers may move horizontally, but they do not change rows. 

Additionally, notice that through ladder moves, the underlying permutation $\mathbf{a}$ of $\mathbf{P}$ does not change. 

\begin{figure}[h!]
 \centering
    \begin{subfigure}[b]{0.4\textwidth}
        \centering
        \includegraphics[scale = 0.25]{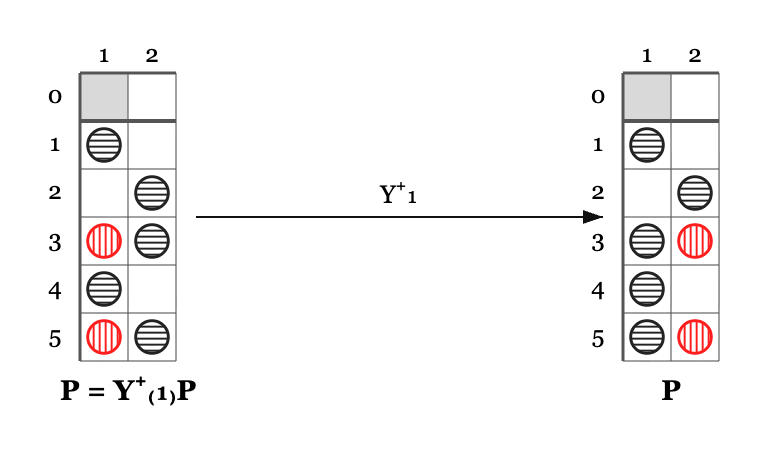}
        \caption{$1$-ladder $Y^+_i$ move}
        \label{fig:Y^+ e.g. 1}
    \end{subfigure}
 \quad\quad
    \begin{subfigure}[b]{0.4\textwidth}
        \centering
        \includegraphics[scale = 0.25]{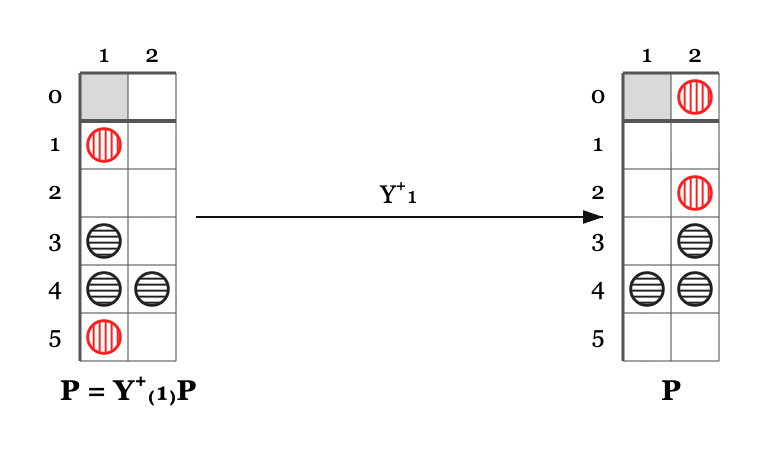}
        \caption{Big Ladder $Y^+_i$ Move}
        \label{fig:Y^+ e.g. 2}
    \end{subfigure}

    \caption{Types of Ladder Moves}
\end{figure}
    
For \(i\in \mathbb Z\), we define
\[
    Y_{\geq i}^+
    :=
    Y_i^+Y_{i+1}^+Y_{i+2}^+\cdots .
\]
When applied to a fixed super pipe dream, $Y_{\geq i}^+$ is a finite composition.
It moves all red checkers in columns weakly to the right of \(i\) one
column to the right.

We also define
\[
    Y^+
    :=
    \cdots Y_{-1}^+Y_0^+Y_1^+\cdots .
\]
This composition is also finite when applied to any fixed super pipe dream and
flows every red checker one column to the right.

Finally, for integers \(i\leq m\), suppose that \(\mathbf P\) has no red
checkers in columns ranging from $i+1$ to $m+1$. Define
\[
    Y_i^{+m}\mathbf P
    :=
    \bigl(Y_m^+Y_{m-1}^+\cdots Y_i^+\bigr)\mathbf P .
\]
This composition successively flows the red checkers originally in
column \(i\) until they reach column \(m+1\).

\begin{figure}[H]
    \centering
    \includegraphics[scale = 0.18]{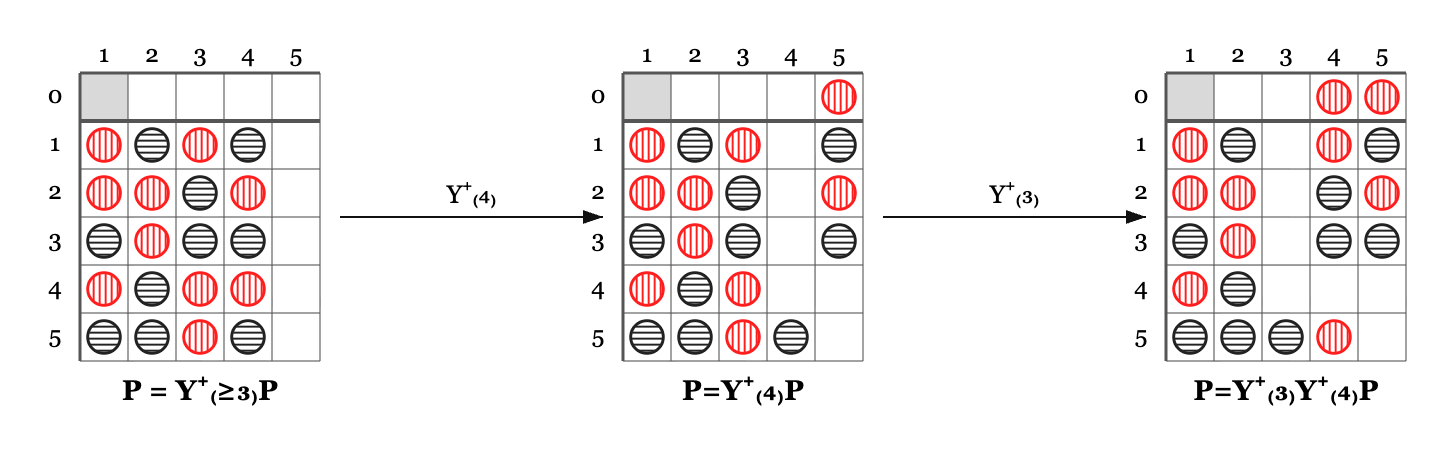}
    \caption{Illustration of $Y^+_{\geq 3}$}
    \label{fig:Y^+_3}
\end{figure}

\begin{figure}[H]
    \centering
    \includegraphics[scale = 0.18]{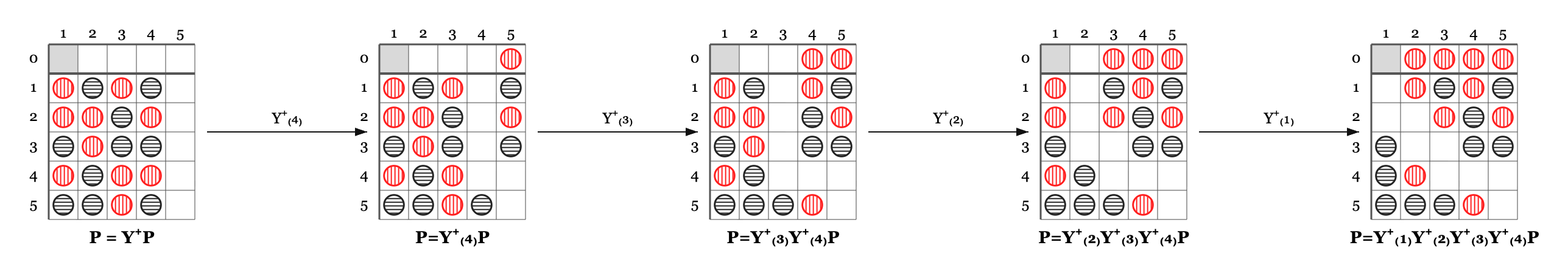}
    \caption{Illustration of $Y^+$}
    \label{fig:Y^+}
\end{figure}

Repeated applications of \(Y^+\) eventually separate the red and black
components of a super pipe dream. More precisely, for a super pipe dream
\(\mathbf P\), choose \(N\) sufficiently large so that, in
\[
    \mathbf P'=(Y^+)^N\mathbf P,
\]
all red checkers lie northeast of all black checkers. Let \(V\) be the
ordinary pipe dream determined by the black component of \(\mathbf P'\),
and let \(U\) be the ordinary pipe dream obtained from the red component
using the transpose-and-shift convention of
\cite[Section~6]{dennin2025cauchy}. We then define
\[
    \operatorname{Rect}(\mathbf P)=(V,U).
\]
Thus rectification separates a super pipe dream into two ordinary pipe
dreams corresponding to its black and red components.

\begin{figure}[H]
    \centering
    \includegraphics[scale = 0.3]{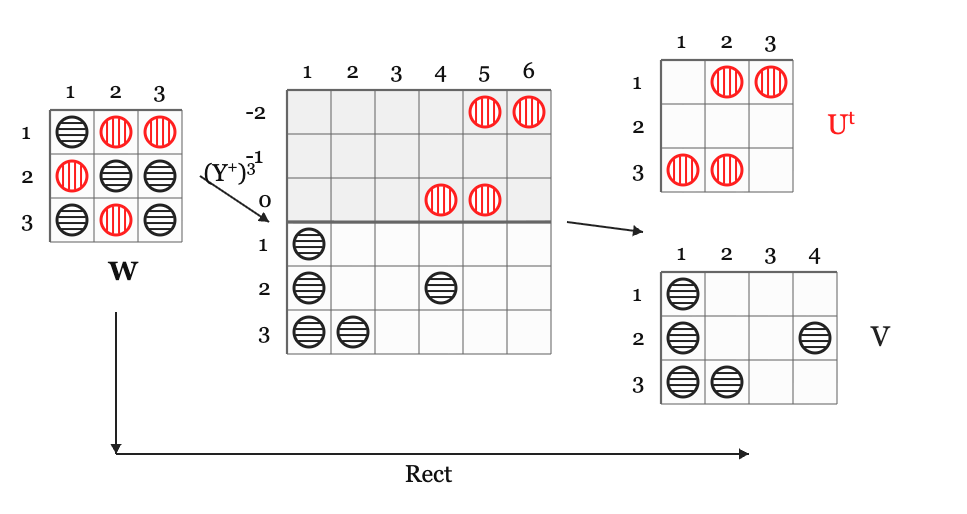}
    \caption{Example of pipe dream rectification.}
    \label{fig:placeholder}
\end{figure}

\subsection{Insertion and recording}

Let \(V\) be a reduced pipe dream for a Grassmannian permutation, visualized as a tiling in which checker locations correspond to crossing tiles and unoccupied locations correspond to elbow tiles. Following \cite[Section~8.1]{dennin2025cauchy}, let \(\operatorname{tab}(V)\) denote the reverse semistandard Young tableau associated with \(V\).

\begin{example}

Suppose we have a reduced pipe dream $V = \begin{bmatrix}
    0 & \cdot \\ 0 & \cdot \\ 0 & 0
\end{bmatrix}$, we can find $\operatorname{tab}(V)$ by first seeing which pipes the checker belongs on. Checkers on the same pipe belong in the same row in $\operatorname{tab}(V)$, and rows are organized in strictly decreasing order. 
\begin{figure} [H]
    \centering
    \includegraphics[scale = 0.18]{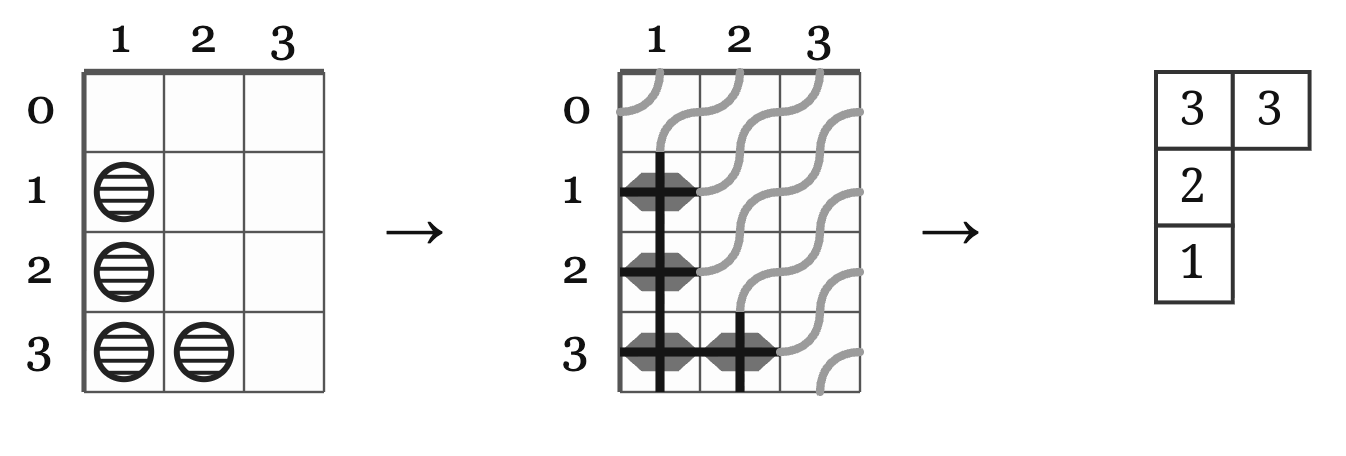}
    \caption{Reduced Pipe Dream $V$ in $tab(V)$}
    \label{fig:tab}
\end{figure}
In other words, the entries corresponding to the crossings along a fixed horizontal pipe form a row of \(\operatorname{tab}(V)\), and each crossing is recorded by its row index in the pipe dream. The resulting tableau is weakly decreasing along rows and strictly decreasing down columns.
\end{example}

Let
\[
    A\in \mathrm{BM}_{m\times n},
\]
and identify \(A\) with its corresponding super pipe dream. If
\[
    \operatorname{Rect}(A)=(V,U),
\]
we define the \textit{insertion tableau} of \(A\) by
\[
    \operatorname{ins}(A):=\operatorname{tab}(V).
\]

The \textit{recording tableau} records which column of \(A\) creates each box of
\(\operatorname{ins}(A)\).

For a tableau \(X\), let \(\operatorname{shape}(X)\) denote its Young
diagram. If \(Y\subseteq X\), let \(X\setminus Y\) denote the skew
subtableau consisting of the boxes of \(X\) that do not belong to \(Y\).
If every entry of a tableau \(Y\) is equal, let \(\operatorname{number}(Y)\)
denote this common entry.

Define nested tableaux
\[
    T_n\subseteq T_{n-1}\subseteq\cdots\subseteq T_1
\]
by requiring that
\[
    \operatorname{shape}(T_i)
    =
    \operatorname{shape}\bigl(\operatorname{ins}(A_i')\bigr)^t
    \qquad\text{for each }i\in[n],
\]
together with the labeling conditions
\[
    \operatorname{number}(T_n)=n
\]
and
\[
    \operatorname{number}(T_{i-1}\setminus T_i)=i-1
    \qquad\text{for }2\leq i\leq n.
\]
We then define
\[
    \operatorname{rec}(A):=T_1.
\]

\begin{example}
    Let $A = \begin{bmatrix}
        1 & 0 & 1 \\ 0 & 1 & 1
    \end{bmatrix}$.\\
    In this case, $n =3$ so let's start with $T_3$. Here, $A'_3 = \begin{bmatrix}
        1 \\ 1
    \end{bmatrix}$ so $ins(A'_3) = \begin{ytableau}
        2 \\ 1
    \end{ytableau}$. Because $\operatorname{shape}(T_3) = \operatorname{shape}(\operatorname{ins}(A'_3))^t$ and $\operatorname{number}(T_3) = 3$, $T_3 = \begin{ytableau}
        3 & 3
    \end{ytableau}$. \\
    Next, let's find $T_2$. $A'_2 = \begin{bmatrix}
        0 & 1 \\ 1 & 1
    \end{bmatrix}$, rectifying, we get $\operatorname{ins}(A'_2) = \begin{ytableau}
        2 & 2 \\ 1
    \end{ytableau}$. By $\operatorname{shape}(T_3) = \operatorname{shape}(\operatorname{ins}(A'_3))^t$ and $\operatorname{number}(T_{i-1}\setminus T_i)=i-1$, we get $T_2 = \begin{ytableau}
        3 & 3 \\ 2
    \end{ytableau}$. \\
    Finally, let's find $T_1 = rec(A)$. $A'_1 = A = \begin{bmatrix}
        1 & 0 & 1 \\ 0 & 1 & 1
    \end{bmatrix}$, rectifying as shown in Figure \ref{fig:rectninsA}, we get $\operatorname{ins}(A'_2) = \begin{ytableau}
        2 & 2 \\ 1 & 1
    \end{ytableau}$. Hence, $T_1 = rec(A) =\begin{ytableau}
        3 & 3 \\ 2 & 1
    \end{ytableau}$. 
    Notice that $T_n\subseteq T_{n-1}\subseteq\cdots\subseteq T_1$. 
    \begin{figure}[H]
        \centering
        \includegraphics[scale = 0.3]{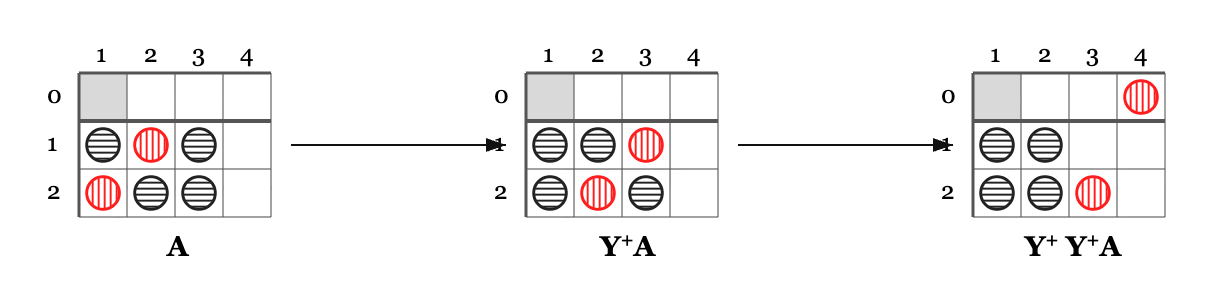}
        \includegraphics[scale = 0.3]{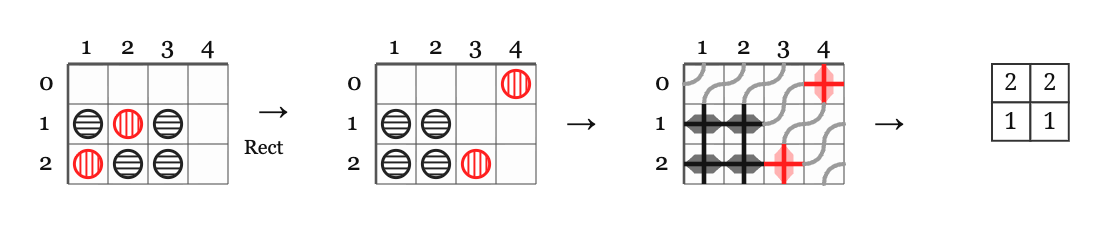}
        \caption{Exact Rectification and Insertion of $A$}
        \label{fig:rectninsA}
    \end{figure}
\end{example}

\subsection{Complementary and full tableaux}
Let \(T\) be a tableau created by some $A \in BM_{m\times n}$.

Define \(\mathbf{T}\) to be the \emph{full tableau} if \(\mathbf{T}\) has dimensions $m \times n$ with each row containing entries \(i\in[n]\) and each column containing entries \(j\in[m]\), arranged according to the reverse semistandard convention. \\

We define the complement tableau \(\overline{T}\) by reversing the order of the columns of \(T\) and taking the set-theoretic complement of the entries in each column.

More precisely, for each \(i\in[n]\),
\[
    \operatorname{col}_i(\overline{T})
    :=
    [m]\setminus
    \operatorname{col}_{m+1-i}(T),
\]
where \(\operatorname{col}_i(T)\) denotes the set of entries appearing
in the \(i\)-th column of \(T\).

Thus, for every \(i\in[n]\), the union of the \((m+1-i)\)-th column of \(T\) and the \(i\)-th column of \(\overline{T}\) gives the corresponding column of the full tableau \(\mathbf{T}\). In other words,
\[
    \operatorname{col}_{m+1-i}(T)
    \cup
    \operatorname{col}_i(\overline{T})
    =
    [m].
\]

\begin{example}
    For $T$ created by $A \in BM_{4 \times 4}$, if $T = \begin{ytableau}
        4 & 4 & 3 \\ 3 & 2 & 1 \\ 1 
    \end{ytableau}$, then $\overline{T} = \begin{ytableau}
        4 & 4 & 3 & 2 \\ 3 & 2 & 1\\ 2  \\ 1
    \end{ytableau} $. 
    
\end{example}

\section{Proof sketch of Theorem \ref{thm:main-thm}}

Let us illustrate the proof idea through an example.

Let $A = \begin{bmatrix}
    1 & 0 & 1 \\ 1 & 1 & 0 \\
    0 & 1 & 1 \\ 1 & 0 & 1 
\end{bmatrix} \in BM_{4 \times 3}$.We want to show $\overline{\rec(A)} = \ins(A^\dagger)$.

Let
\[
A_1 = A, \quad A_2 = \begin{bmatrix}
    0 & 0 & 1 \\ 0 & 1 & 0 \\
    0 & 1 & 1 \\ 0 & 0 & 1 
\end{bmatrix}, \quad A_3 = \begin{bmatrix}
    0 & 0 & 1 \\ 0 & 0 & 0 \\
    0 & 0 & 1 \\ 0 & 0 & 1  
\end{bmatrix}.
\]
and $A'_i$ be $A_i$ with the first $i-1$ columns removed. Similarly, let $R_i = A_i^\dagger$ or \[R_1 = R = \begin{bmatrix}
    0 & 0 & 1 & 0 \\
    1 & 0 & 0 & 1 \\ 0 & 1 & 0 & 0
\end{bmatrix}, \quad R_2 = \begin{bmatrix}
    1 & 1 & 1 & 1 \\
    1 & 0 & 0 & 1 \\ 0 & 1 & 0 & 0
\end{bmatrix}, \quad R_3 = \begin{bmatrix}
     1 & 1 & 1 & 1 \\
    1 & 1 & 1 & 1 \\ 0 & 1 & 0 & 0
\end{bmatrix}\] and $R_i'$ be $R_i$ but with the first $i-1$ rows replaced with $\cdot$ meaning we shift the do not place any checker in those rows in the pipe dream notation.

We will show $\overline{\rec(A)} = \ins(A^\dagger)$ by induction.

For the base case,  Figure \ref{fig:rectA_4} shows that $\ins(A'_3) = \begin{ytableau}
        4 \\ 3 \\ 1
    \end{ytableau}$, meaning $T_3 =
    \begin{ytableau}
        3 & 3 & 3
    \end{ytableau}$.

\begin{figure}[H]
    \centering
    \includegraphics[scale = 0.3]{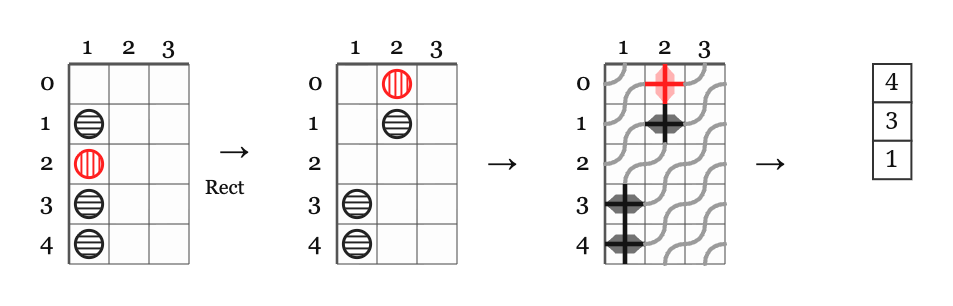}
    \caption{Rectification and Tableaux of $A'_3$}
    \label{fig:rectA_4}
\end{figure}

Similarly, one can check that we also have $\rec(A_3) = \begin{ytableau}
        3 & 3 & 3
    \end{ytableau} = T_3$.

On the other hand, Figure \ref{fig:RectR'_4} shows that $\ins(R'_3) = \begin{ytableau}
        3
    \end{ytableau}$ and hence $\ins(R_3) = \fill(\ins(R'_3)) = \begin{ytableau}
        3 & 2 & 2 & 2 \\ 2 & 1 & 1 & 1 \\ 1
    \end{ytableau}$. 

\begin{figure}[H]
    \centering
    \includegraphics[scale = 0.3]{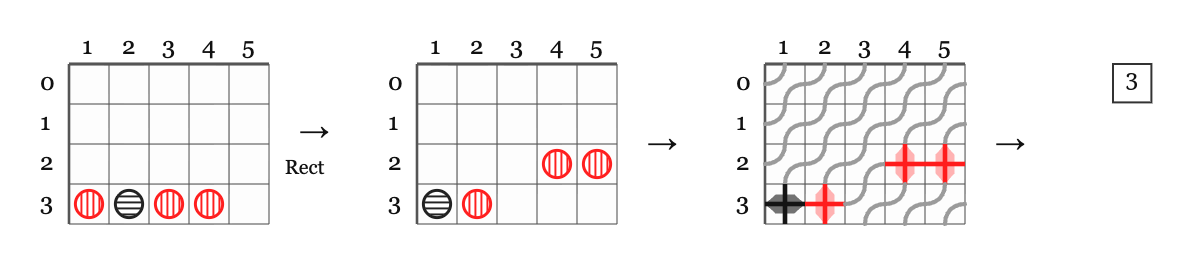}
    \caption{Rectification and Tableaux of $R'_3$}
    \label{fig:RectR'_4}
\end{figure}

Hence, we do have $\overline{\rec(A_3)} = \ins(R_3)$.

Now we continue to $A_2$ and $R_2$.

Figure \ref{fig:A3} shows that $\ins(A'_2) = \begin{ytableau}
        4 & 3 \\
        3 & 1 \\
        2
    \end{ytableau}.$ Using the definition of $\rec()$, since $\operatorname{number}(T_{i-1}\setminus T_i)=i-1$ and $\operatorname{shape}(T_i)
    =
    \operatorname{shape}\bigl(\operatorname{ins}(A_i')\bigr)^t$, we have $ T_2 = \begin{ytableau}
        3 & 3 & 3 \\
        2 & 2
    \end{ytableau}$. 

\begin{figure}[H]
    \centering
    \includegraphics[scale = 0.3]{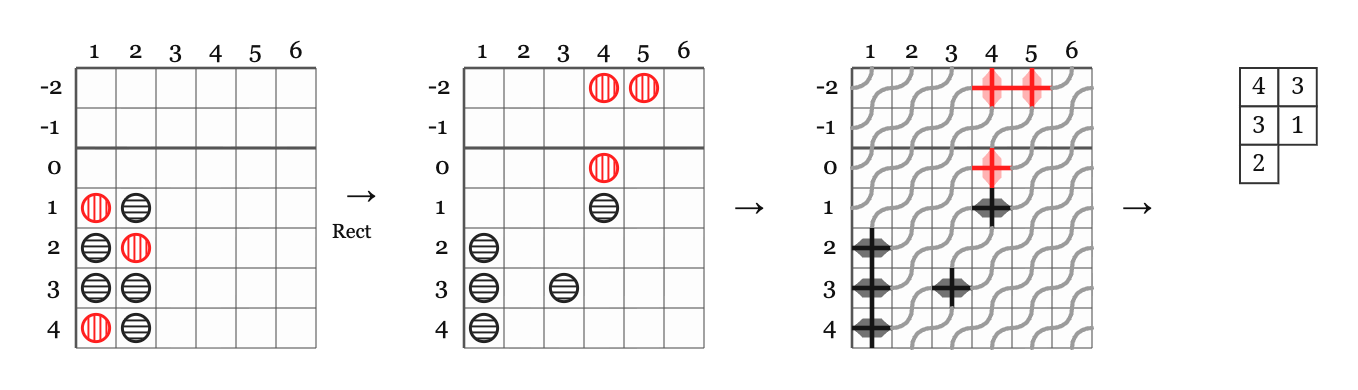}
    \caption{Rectification and Tableaux of $A'_2$}
    \label{fig:A3}
\end{figure}

Similarly, we also check that $ \rec(A_2) = \begin{ytableau}
        3 & 3 & 3 \\
        2 & 2
    \end{ytableau} = T_2$ 
    
On the other hand, Figure \ref{fig:R3} shows that $\ins(R'_2) = \begin{ytableau}
        3 & 2 \\ 2
    \end{ytableau}$, hence, $\ins(R_3) = \fill(\ins(R'_2)) = \begin{ytableau}
        3 & 2 & 1 & 1 \\ 2 & 1 \\ 1
    \end{ytableau}$. 

\begin{figure}[H]
    \centering
    \includegraphics[scale = 0.3]{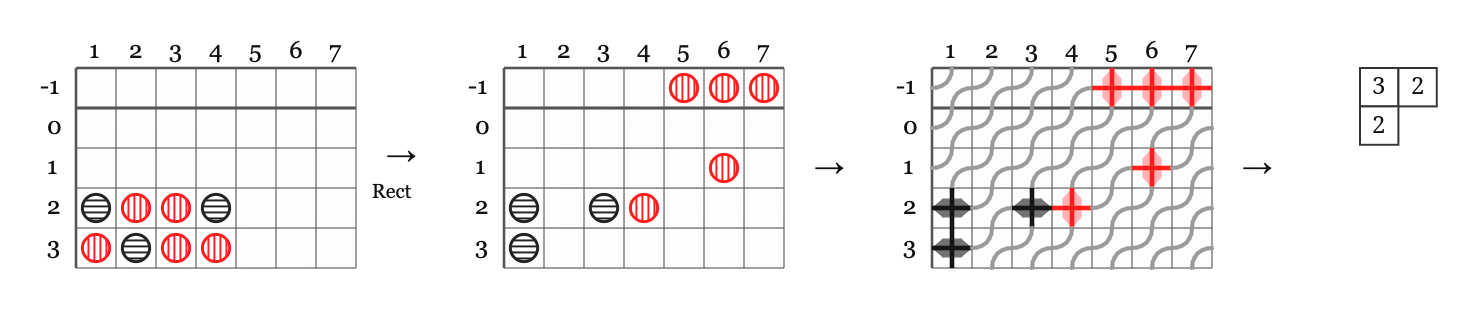}
    \caption{Rectification and Tableaux of $R'_2$}
    \label{fig:R3}
\end{figure}

Therefore, $\overline{\rec(A_2)} = \ins(R_2).$

Observe that to get from $\rec(A_3)$ to $\rec(A_2)$, we added only blocks of $2$'s and the shape of the $3$'s remain the same. This is because we added the second column which can only add $2$'s to the recording tableau by definition and does not change the shape of the recording tableaux of column $3$. Additionally, notice this is exactly the same as how we get from $T_3$ to $T_2$. 

On the other hand, to get from $\ins(R'_3)$ to $\ins(R'_2)$, we also only add $2$'s because black checkers stay in the same row after rectification. Additionally, the $3$'s stay the same because inserting a row checkers above did not interfere with the position after rectification of the black checkers in row $3$.

Furthermore, because the operation $\fill()$ fills all position where $ i \leq 2$, we have $\overline{\rec(A_3)} = \ins(R_3)$. In addition, $\rec(A_2)$ and $\ins(R_2)$ preserves the shapes of $\rec(A_3)$ and $\ins(R_3)$ respectively. Hence, as long as the $2$'s are in complementary positions in $\rec(A_2)$ and $\ins(R_2)$, which is true in this case, then we have $\overline{\rec(A_2)} = \ins(R_2)$.

Let's see this again for $A_1$ and $R_1$. Figure \ref{fig:A2} shows that $\ins(A'_2) = \begin{ytableau}
        4 & 4 & 3\\
        3 & 2 & 1 \\
        2 \\
        1
    \end{ytableau}$. By definition of $\rec()$, we have $\rec(A_1) = T_1 = \begin{ytableau}
        3 & 3 & 3 & 1\\
        2 & 2 \\
        1 & 1
    \end{ytableau}$. Observe that, again, the positions of $2$'s and $3$'s in $T_1$ are the same as in $T_2$.

\begin{figure}[H]
    \centering
    \includegraphics[scale = 0.3]{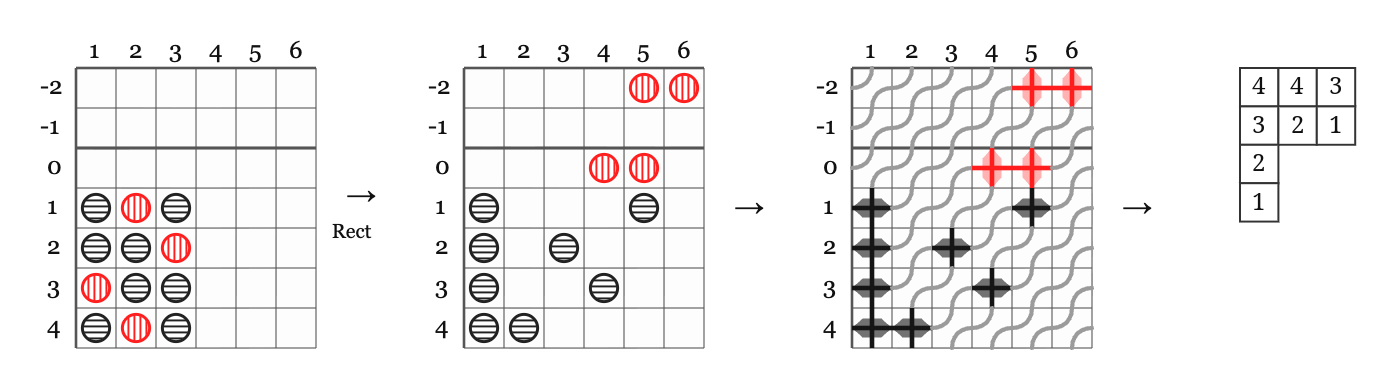}
    \caption{Rectification and Tableaux of $A'_1$}
    \label{fig:A2}
\end{figure}

On the other hand, Figure \ref{fig:R2} shows that $\ins(R'_1) = \ins(R_1) = \begin{ytableau}
        3 & 2 \\ 2 & 1
    \end{ytableau}$. As expected, the positions of $2$'s and $3$'s in $\ins(R_1)$ are the same as in $\ins(R_2)$.

\begin{figure}[H]
    \centering
    \includegraphics[scale = 0.3]{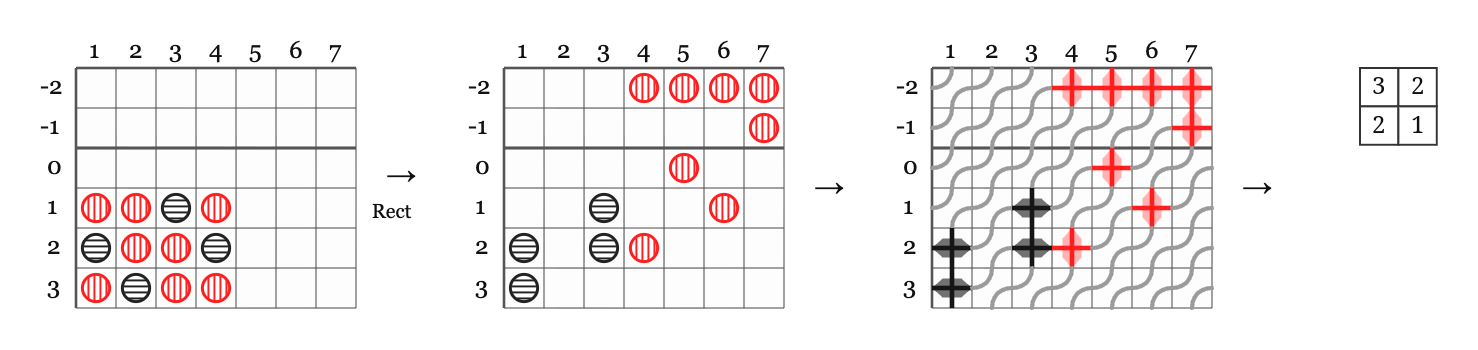}
    \caption{Rectification and Tableaux of $R'_2$}
    \label{fig:R2}
\end{figure}

Therefore, $\overline{\rec(A_1)} = \ins(R_1),$ meaning $\overline{\rec(A)} = \ins(R)$.

By Theorem \ref{prop:base-case}, the base case is true. By Theorem \ref{thm:induction-step} the induction step is true. Hence, the conjecture is true and we have $\overline{rec(A)} = ins(A^\dagger) \ \forall A \in BM_{m \times n}.$

\section{Proof of main theorem}

First, we have a local semicommutativity relation.

\begin{lemma}[Column locality of $Y_j^+$]
\label{lem:Y-locality}
Let $P$ be a super pipe dream with no red checkers in column $j+1$, so that $Y_j^+(P)$ is defined. Then,

\begin{enumerate}
    \item $P$ and $Y_j^+(P)$ agree in every column other than $j$ and $j+1$.

    \item The operation $Y_j^+(P)$ is determined entirely by and affects only columns $j$ and $j+1$.
\end{enumerate}

\end{lemma}

\begin{proof}
By definition, $Y_j^+$ moves the red checkers in column $j$ into column $j+1$. For each red checker chosen during the construction, the associated ladder is a $k\times 2$ rectangle whose two columns are exactly $j$ and $j+1$. Thus every checker moved during the operation remains within these two columns, and every column $q\notin\{j,j+1\}$ is unchanged.

Additionally, the choice and size of each ladder are determined only by the locations of the red and black checkers in columns $j$ and $j+1$. In the $1$-ladder case, one checks whether the position immediately to the right of the chosen red checker is occupied, and, in the big-ladder case, finds the first unoccupied position above it in column $j$. 
\end{proof}

\begin{proposition}[Commutativity of $Y_{j-1}^+$ and $Y_{j+1}^+$]
\label{Commutative1}
Suppose that \(\mathbf P_j\) is a super pipe dream with no red checkers
in columns \(j+1\) and \(j+2\). Then
\[
    Y_j^+Y_{j+1}^+Y_{j-1}^+Y_j^+\mathbf P_j
    =
    Y_j^+Y_{j-1}^+Y_{j+1}^+Y_j^+\mathbf P_j .
\]
Intuitively, both sides of the equation should shift red checkers in columns $j-1$ and $j$ to columns $j+1$ and $j+2$, respectively.
\end{proposition}

\begin{proof}
    First, we check that all operators appearing in the statement are defined. By assumption, columns \(j+1\) and \(j+2\) contain no red checkers. Hence, \(Y_j^+\) is defined on \(\mathbf P_j\). After applying \(Y_j^+\), column \(j\) contains no red checkers, so \(Y_{j-1}^+\) is defined. Also, column \(j+2\) still contains no red checkers, so \(Y_{j+1}^+\) is defined. After \(Y_{j+1}^+\), \(Y_j^+\) is defined. Therefore, all operators in the statement is defined.
    
    Set
    \[
        \mathbf P_j' := Y_j^+\mathbf P_j .
    \]
    It suffices to show that
    \[
        Y_{j+1}^+Y_{j-1}^+\mathbf P_j'
        =
        Y_{j-1}^+Y_{j+1}^+\mathbf P_j',
    \]
    because then applying the same operator \(Y_j^+\) to both sides gives the desired equality.

    By lemma \ref{lem:Y-locality}, \(Y_{j-1}^+\) only acts on the $j-1\text{ and } j$ while \(Y_{j+1}^+\) only acts on columns $j+1 \text{ and } j+2$ making these two operations disjoint and commutable. Hence,
    \[
        Y_{j+1}^+Y_{j-1}^+\mathbf P_j'
        =
        Y_{j-1}^+Y_{j+1}^+\mathbf P_j'.
    \]
    Applying \(Y_j^+\) to both sides yields
    \[
        Y_j^+Y_{j+1}^+Y_{j-1}^+Y_j^+\mathbf P_j
        =
        Y_j^+Y_{j-1}^+Y_{j+1}^+Y_j^+\mathbf P_j .
    \]
    Therefore, the proposition follows.
\end{proof}

The relation in Proposition \ref{Commutative1} can be generalized as follows.

\begin{proposition}\label{gencom}
    Suppose that \(\mathbf P_{i,j}\) is a super pipe dream on which both \(Y_i^+\) and \(Y_j^+\) are defined. If
    \[
        |i-j|\geq 2,
    \]
    then
    \[
        Y_i^+Y_j^+\mathbf P_{i,j}
        =
        Y_j^+Y_i^+\mathbf P_{i,j}.
    \]
\end{proposition}
\begin{proof}
    By lemma \ref{lem:Y-locality}, the operator \(Y_i^+\) acts only on columns $i$ and $i+1$ while \(Y_j^+\) acts only on columns $j$ and $j+1$. Since \(|i-j|\geq 2\), these two pairs of columns are disjoint and commutative. Hence, 
    \[
        Y_i^+Y_j^+\mathbf P_{i,j}
        =
        Y_j^+Y_i^+\mathbf P_{i,j}.
    \]
\end{proof}

\begin{figure}[H]
    \centering
    \includegraphics[width=0.8\linewidth]{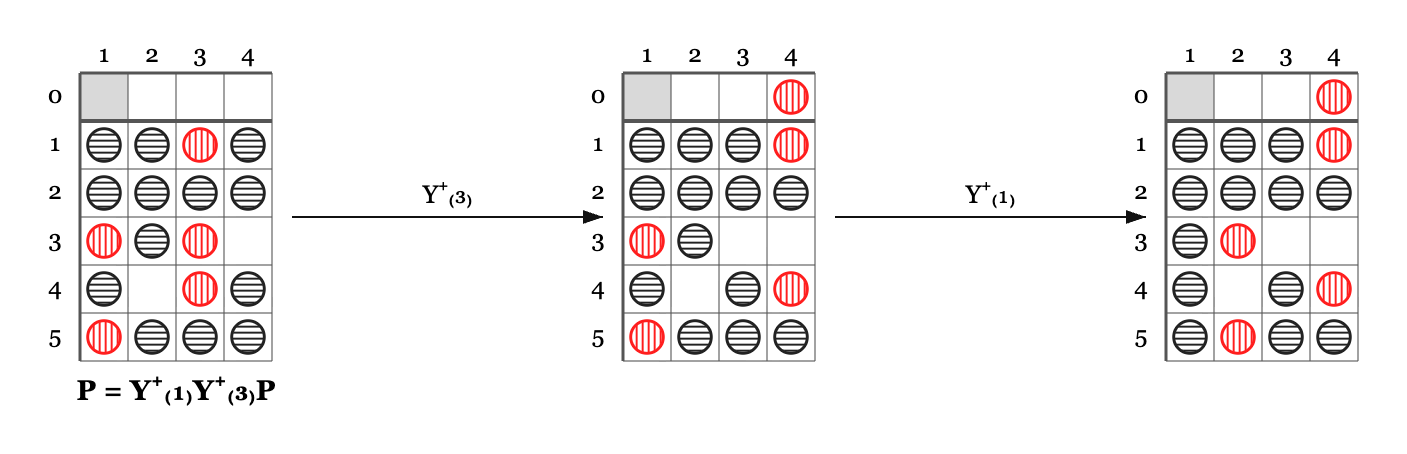}
    \vspace{0.5cm}
    \includegraphics[width=0.8\linewidth]{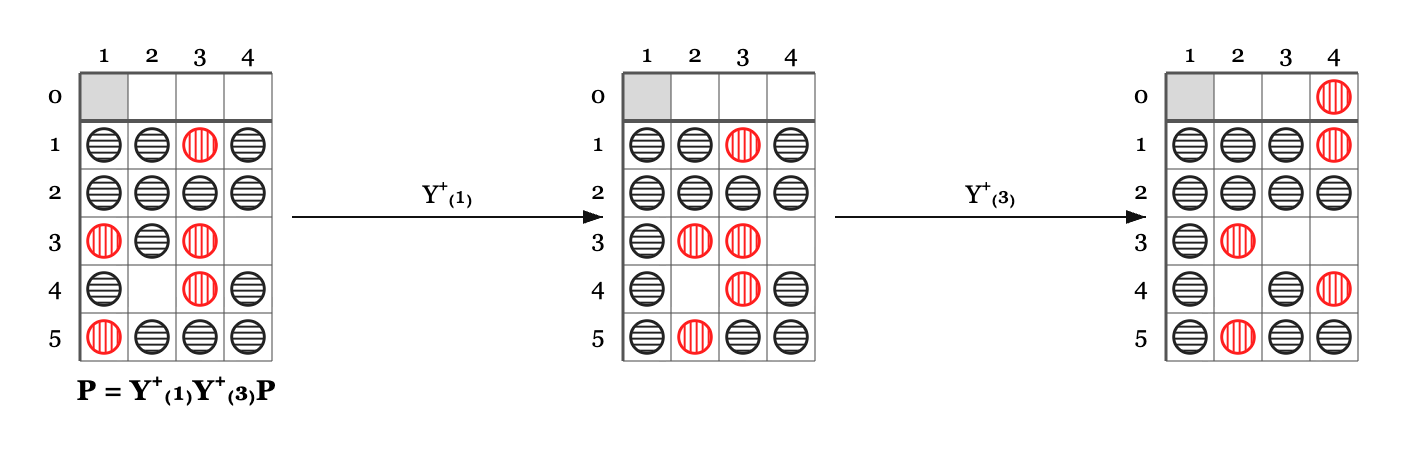}
    \caption{Illustration of a reference like Proposition \ref{Commutative1}}
    \label{fig:commute}
\end{figure}

\begin{theorem}[A semicommutativity relation for \(Y^+\)]
\label{Commutative2}
Let \(\mathbf P_{j,k}\) be a super pipe dream whose red checkers are contained in columns $ j,$ $j+1, $ $\dots, $ $k$ where \(j\leq k\). Then
\[
    Y_{\geq k}^+Y_{\geq k-1}^+\cdots
    Y_{\geq j}^+\mathbf P_{j,k}
    =
    Y_j^{+k}Y_{j+1}^{+(k+1)}\cdots
    Y_k^{+(2k-j)}\mathbf P_{j,k}.
\]
Intuitively, the left-hand side performs successive global flows beginning with \(Y_{\geq j}^+\) and ending with \(Y_{\geq k}^+\). The right-hand side instead moves the red checkers in each original column separately \(k-j+1\) times, processing the columns from \(k\) down to \(j\). In both descriptions, every red checker is moved \(k-j+1\) columns to the right and therefore ends strictly to the right of column \(k\).
\end{theorem}

\begin{proof}
Expand each operator \(Y_{\geq r}^+\) as
\[
    Y_{\geq r}^+
    =
    Y_r^+Y_{r+1}^+Y_{r+2}^+\cdots .
\]
Since \(\mathbf P_{j,k}\) has no red checkers outside columns \(j,\ldots,k\), only finitely many factors act nontrivially. Additionally, after $r$ global flows, every red checker has moved exactly $r$ columns to the right. Therefore, before the $s$-th global flow the rightmost possible red column is $k + s - 1$, so all factors beyond the corresponding upper index act trivially.\\
Therefore, the left-hand side can be written as
\begin{align*}
    &Y_{\geq k}^+Y_{\geq k-1}^+\cdots Y_{\geq j}^+\mathbf P_{j,k} \\
    &=
    \bigl(Y_k^+Y_{k+1}^+Y_{k+2}^+\cdots Y_{2k-j}^+\bigr)
    \bigl(Y_{k-1}^+Y_k^+Y_{k+1}^+\cdots Y_{2k-j-1}^+\bigr)
    \cdots
    \bigl(Y_j^+Y_{j+1}^+Y_{j+2}^+\cdots Y_k^+\bigr)
    \mathbf P_{j,k}.
\end{align*}

Now we use Proposition~\ref{gencom} to commute operators whose indices differ by at least \(2\). This allows us to move the first operator from each parenthesized block next to each other, then the second operator from each block next to each other, and so on. In other words, we reorder the above product as
\begin{align*}
    &\bigl(Y_k^+Y_{k-1}^+\cdots Y_j^+\bigr)
    \bigl(Y_{k+1}^+Y_k^+\cdots Y_{j+1}^+\bigr)
    \cdots
    \bigl(Y_{2k-j}^+Y_{2k-j-1}^+\cdots Y_k^+\bigr)
    \mathbf P_{j,k}.
\end{align*}

This reordering is valid because, whenever an operator is moved past another operator during this process, their indices differ by at least \(2\). Hence, the two operators commute by Proposition~\ref{gencom}.

By the definition
\[
    Y_i^{+m}
    =
    Y_m^+Y_{m-1}^+\cdots Y_i^+,
\]
we have
\[
    Y_k^+Y_{k-1}^+\cdots Y_j^+
    =
    Y_j^{+k},
\]
\[
    Y_{k+1}^+Y_k^+\cdots Y_{j+1}^+
    =
    Y_{j+1}^{+(k+1)},
\]
and continuing in the same way,
\[
    Y_{2k-j}^+Y_{2k-j-1}^+\cdots Y_k^+
    =
    Y_k^{+(2k-j)}.
\]
Therefore
\[
    Y_{\geq k}^+Y_{\geq k-1}^+\cdots Y_{\geq j}^+\mathbf P_{j,k}
    =
    Y_j^{+k}Y_{j+1}^{+(k+1)}\cdots
    Y_k^{+(2k-j)}\mathbf P_{j,k}.
\]
\end{proof}

\begin{figure}[H]
    \centering
    \includegraphics[scale = 0.18]{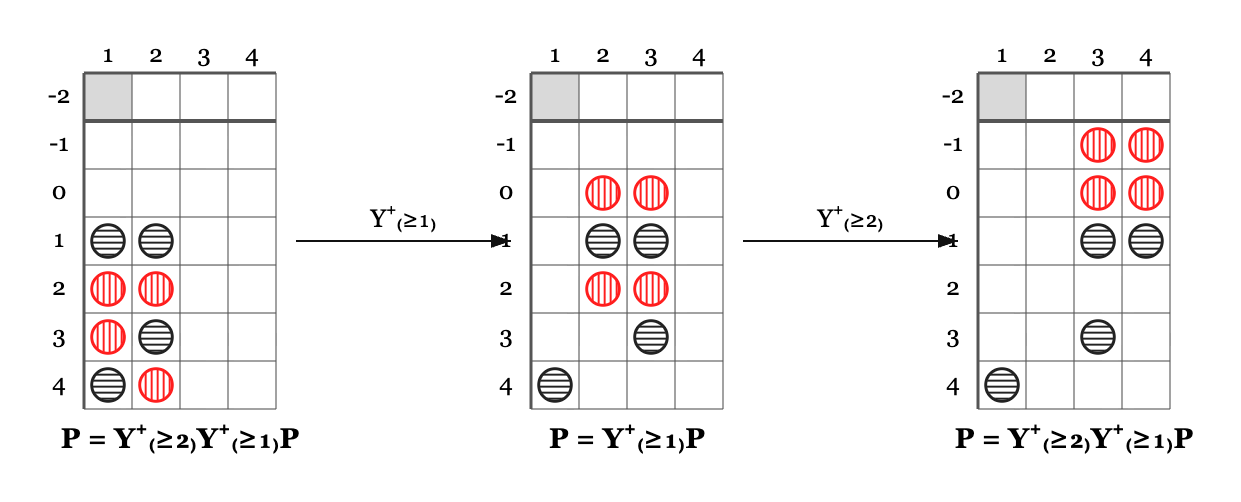}
    \caption{Illustration of $Y^+_{\geq 2}Y^+_{\geq 1}\mathbf P_{1,2}$}
    \label{fig:rightsiderect}
\end{figure}
\begin{figure}[H]
    \centering
    \includegraphics[scale = 0.18]{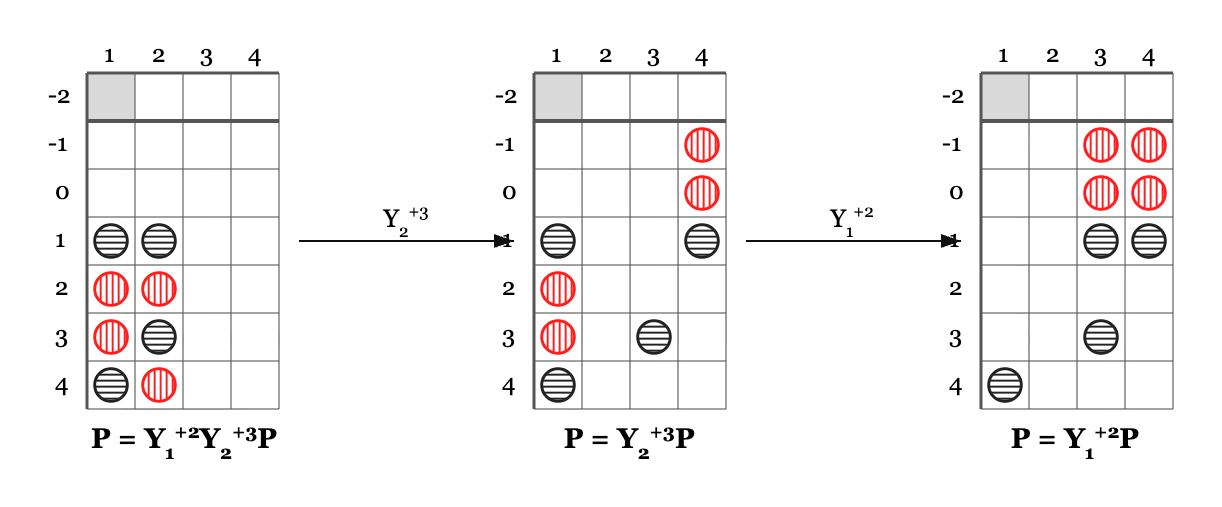}
    \caption{Illustration of $Y^{+2}_{1}Y^{+3}_{2}\mathbf P_{1,2}$}
    \label{fig:leftsiderect}
\end{figure}

\begin{example}
    Here we can see that both figure \ref{fig:rightsiderect} and figure \ref{fig:leftsiderect} uses $\mathbf P_{1,2}$. While figure \ref{fig:rightsiderect} uses $Y^+_{\geq 2}Y^+_{\geq 1}\mathbf P_{1,2}$ and figure \ref{fig:leftsiderect} uses $Y^{+2}_{1}Y^{+3}_{2}\mathbf P_{1,2}$, they both give the same rectified pipe dream at the end. 
\end{example}

Let
\[
\mathbf 0_{m\times n}
=
\begin{bmatrix}
0 & \cdots & 0\\
\vdots & \ddots & \vdots\\
0 & \cdots & 0
\end{bmatrix}.
\]
Let
\[
    A=
    \begin{bmatrix}
        C_1 & C_2 & \cdots & C_n
    \end{bmatrix}
    \in \mathrm{BM}_{m\times n},
\]
where \(C_1,\ldots,C_n\) are the columns of \(A\). For each \(i\in[n]\),
define the truncated suffix matrix
\[
    A_i'
    :=
    \begin{bmatrix}
        C_i & C_{i+1} & \cdots & C_n
    \end{bmatrix}
    \in \mathrm{BM}_{m\times(n-i+1)},
\]
and define the padded suffix matrix
\[
    A_i
    :=
    \begin{bmatrix}
        \mathbf 0_{m\times(i-1)} & C_i & \cdots & C_n
    \end{bmatrix}
    \in \mathrm{BM}_{m\times n}.
\]
Thus \(A_1=A\), while \(A_n\) remembers only the last column of \(A\).

\begin{theorem}
\label{any-order}
Let
\[
    \operatorname{Rect}(A_i)=(V_i,U_i)
    \qquad\text{and}\qquad
    \operatorname{Rect}(A_i')=(V_i',U_i').
\]
Let
\[
    M_i
    :=
    \begin{bmatrix}
        \mathbf 0_{m\times(i-1)} & V_i'
    \end{bmatrix},
\] where the first $i-1$ columns of $M_i$ are only red checkers and the columns $\geq i$ are the black checkered only rectified $V_i'$ shifted from starting at column $1$ to starting at column $i$.
Write
\[
    \operatorname{Rect}(M_i)=(V_i^*,U_i^*).
\]
Then
\[
    V_i^*=V_i .
\]
\end{theorem}

\begin{proof}
We use the convention that products of operators act from right to left.
By Theorem~\ref{Commutative2}, the rectification of \(A_i\) may be
performed by moving the red checkers column-by-column:
\[
    Y_{\geq n}^+Y_{\geq n-1}^+\cdots Y_{\geq 1}^+ A_i
    =
    Y_1^{+n}Y_2^{+(n+1)}\cdots Y_n^{+(2n-1)}A_i.
\]
Thus the black component \(V_i\) of \(\operatorname{Rect}(A_i)\) is the
black component obtained after applying
\[
    Y_1^{+n}Y_2^{+(n+1)}\cdots Y_n^{+(2n-1)}
\]
to \(A_i\).

Now split this product into two parts:
\[
    Y_1^{+n}Y_2^{+(n+1)}\cdots Y_n^{+(2n-1)}
    =
    \bigl(Y_1^{+n}Y_2^{+(n+1)}\cdots Y_{i-1}^{+(n+i-2)}\bigr)
    \bigl(Y_i^{+(n+i-1)}\cdots Y_n^{+(2n-1)}\bigr).
\]
The second factor
\[
    Y_i^{+(n+i-1)}\cdots Y_n^{+(2n-1)}
\]
rectifies the part of \(A_i\) coming from the suffix columns
\[
    C_i,C_{i+1},\ldots,C_n.
\]
Since the first \(i-1\) columns of \(A_i\) are zero columns, they contain no
black checkers. Therefore this suffix rectification produces the same black
component as rectifying \(A_i'\), namely \(V_i'\).

In other words, after applying
\[
    Y_i^{+(n+i-1)}\cdots Y_n^{+(2n-1)}
\]
to \(A_i\), the black component is \(V_i'\), embedded in the original
columns. The red checkers coming from the suffix have been moved to columns
strictly to the right of the region that will be affected by the remaining
operators
\[
    Y_1^{+n}Y_2^{+(n+1)}\cdots Y_{i-1}^{+(n+i-2)}.
\]
Hence these already-rectified suffix red checkers cannot affect the
subsequent motion of black checkers. Therefore, for the purpose of computing
the final black component, we may replace the intermediate diagram by
\[
    M_i=
    \begin{bmatrix}
        \mathbf 0_{m\times(i-1)} & V_i'
    \end{bmatrix}.
\]

It remains to rectify the red checkers in the first \(i-1\) columns. By
Theorem~\ref{Commutative2}, this is done by applying
\[
    Y_1^{+n}Y_2^{+(n+1)}\cdots Y_{i-1}^{+(n+i-2)}
\]
to \(M_i\). Thus the black component \(V_i^*\) of \(\operatorname{Rect}(M_i)\)
is the same as the black component obtained from
\[
    \bigl(Y_1^{+n}Y_2^{+(n+1)}\cdots Y_{i-1}^{+(n+i-2)}\bigr)
    \bigl(Y_i^{+(n+i-1)}\cdots Y_n^{+(2n-1)}\bigr)A_i.
\]
But this product is exactly
\[
    Y_1^{+n}Y_2^{+(n+1)}\cdots Y_n^{+(2n-1)}A_i,
\]
whose black component is \(V_i\). Therefore
\[
    V_i^*=V_i.
\]
\end{proof}

\begin{example}
    As an example, let $A_i = \begin{bmatrix}
        0 & 1 & 0 \\ 0 & 0 & 1 \\ 0 & 1 & 1
    \end{bmatrix}$ and $A_i' = \begin{bmatrix}
        1 & 0 \\ 0 & 1 \\ 1 & 1
    \end{bmatrix}$. Rectifying $A_i$ and $A_i'$, we get $V_i = \begin{bmatrix}
        1 & \cdot \\ 1 & \cdot \\ 1 & 1
    \end{bmatrix}$ and $V_i' = \begin{bmatrix}
        1 & \cdot \\ 1 & \cdot \\ 1 & 1
    \end{bmatrix}$. The exact steps can be seen in figure \ref{fig:eg}.
    \begin{figure}[H]
        \centering
        \includegraphics[width=0.75\linewidth]{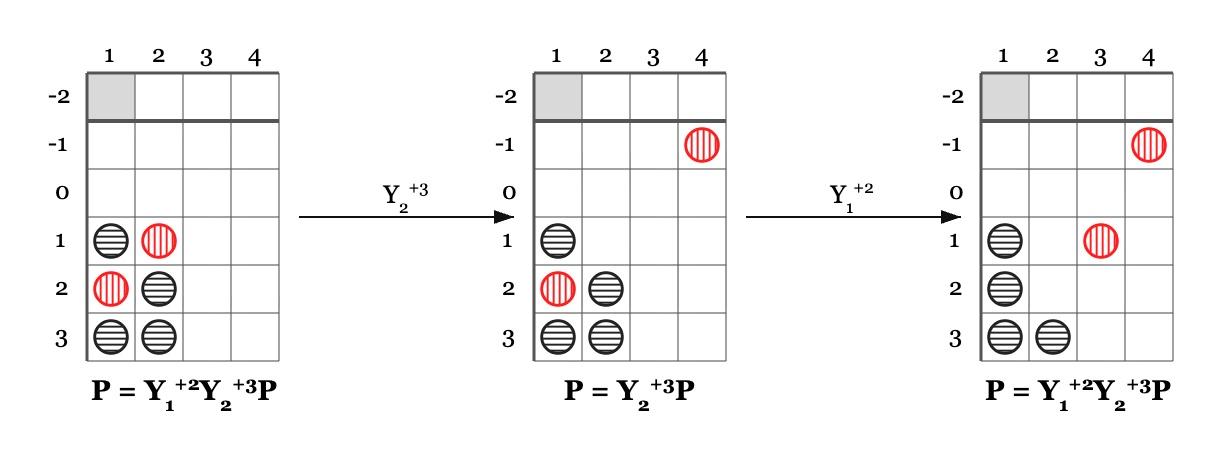}
        \caption{Rectification of $A_i'$}
        \label{fig:eg1}
    \end{figure}
    \begin{figure}[H]
        \centering
        \includegraphics[width=1\linewidth]{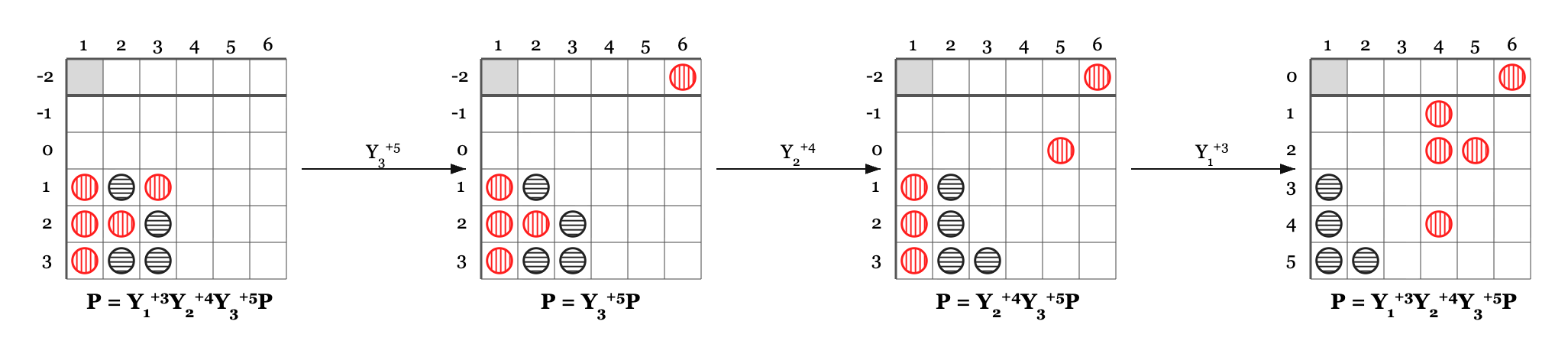}
        \caption{Rectification of $A_i$}
        \label{fig:eg}
    \end{figure}
    We can see that columns $1$ through $3$ in frame three of \ref{fig:eg} shows $M_i$ and rectifying $M_i$ performed the last $Y^{+3}_1$ move for $A_i$, therefore giving $V_i^* = V_i$. 
\end{example}

For a binary column $C_j$ of a matrix $A\in BM_{m\times n}$, define its associated column tableau by
\[
    \mathcal{C}_j
    :=
    \{r\in[m] : A_{rj}=1\},
\]
written in decreasing order according to the reverse semistandard convention.
In particular, if $C_j=0$, then $\mathcal{C}_j=\varnothing$.

\begin{theorem}
\label{insertion-eq}
For
\[
    A_i
    =
    \begin{bmatrix}
        0_{m\times(i-1)} & C_i & C_{i+1} & \cdots & C_n
    \end{bmatrix}
\]
and
\[
    A_i'
    =
    \begin{bmatrix}
        C_i & C_{i+1} & \cdots & C_n
    \end{bmatrix},
\]
we have
\[
    \ins(A_i)
    =
    \ins(A_i').
\]
\end{theorem}

\begin{proof}
    By the column-insertion description of $\operatorname{ins}$,
    the insertion tableau of a binary matrix with associated column tableaux
    $\mathcal{C}_1,\ldots,\mathcal{C}_n$ is
    \[
        \operatorname{ins}(A)
        =
        \mathcal{C}_1 *
        \bigl(
            \mathcal{C}_2 *
            \bigl(
                \cdots *
                (\mathcal{C}_n * \varnothing)
                \cdots
            \bigr)
        \bigr).
    \]
    This agrees with the pipe-dream definition of insertion by
    \cite[Proposition~8.7 and the discussion preceding Proposition~8.8]{dennin2025cauchy}.
    
    For $A_i$, the first $i-1$ binary columns are zero columns. Hence their associated column tableaux are empty:
    \[
        \mathcal{C}_1=\cdots=\mathcal{C}_{i-1}=\varnothing.
    \]
    Since inserting an empty column does nothing,
    \[
        \varnothing * T = T
    \]
    for every tableau $T$. Therefore
    \begin{align*}
        \operatorname{ins}(A_i)
        &=
        \varnothing *
        \Bigl(
            \cdots *
            \bigl(
                \varnothing *
                (\mathcal{C}_i *
                (\mathcal{C}_{i+1} *
                (\cdots *(\mathcal{C}_n*\varnothing)\cdots)))
            \bigr)
            \cdots
        \Bigr)\\
        &=
        \mathcal{C}_i *
        \bigl(
            \mathcal{C}_{i+1} *
            (\cdots *(\mathcal{C}_n*\varnothing)\cdots)
        \bigr).
    \end{align*}
    The final expression is exactly the column-insertion definition of
    $\operatorname{ins}(A_i')$. Hence
    \[
        \operatorname{ins}(A_i)
        =
        \operatorname{ins}(A_i').
    \]
\end{proof}

\begin{cor}
    Using the notation of Theorems \ref{any-order} and \ref{insertion-eq}, we have
    \[
        V_i^* = V_i'.
    \] 
    Intuitively, this corollary says that once the suffix has been rectified to the reduced Grassmannian pipe dream $V_i'$, adding the initial columns of red checkers and rectifying again does not alter the resulting black pipe dream. 
    
    In particular, the additional columns of red checkers may flow through the diagram, but they do not change the connectivity of the existing black pipes of $V_i'$. This conclusion is stronger than preservation of connectedness as after rectification the entire black component is again exactly $V_i'$.
\end{cor}

\begin{proof}
    This directly follows as $\operatorname{tab}(V_i^*) = \operatorname{tab}(V_i) = \operatorname{tab}(V_i')$. Therefore, $V_i^* = V_i'$.
\end{proof}

Recall that $T_i$ is defined such that $\operatorname{shape}(T_i) = \operatorname{shape}\bigl(\operatorname{ins}(A_i')\bigr)^t$ with the labeling conditions $\operatorname{number}(T_n)=n$ and $\operatorname{number}(T_{i-1}\setminus T_i)=i-1 \text{ for }2\leq i\leq n$. 

For a tableau $T$ with entries in $\{i,i+1,\ldots,n\}$, let $T_{\ge r}$ denote the subtableau consisting of the boxes whose entries are in
$\{r,r+1,\ldots,n\}$.

\begin{lemma}[Uniqueness from the growth chain]
\label{lem:growth-chain-uniqueness}
Let
\[
    \lambda_n \subseteq \lambda_{n-1} \subseteq \cdots \subseteq \lambda_i
\]
be a nested sequence of Young diagrams.

There is at most one tableau $T$ with entries in $\{i,i+1,\ldots,n\}$ satisfying
\[
    \operatorname{shape}(T_{\ge r})=\lambda_r
    \qquad
    \text{for every }r=i,\ldots,n.
\]
Equivalently, if such a tableau exists, it is uniquely determined by the growth chain.
\end{lemma}

\begin{proof}
Set
\[
    \lambda_{n+1}:=\varnothing.
\]
For every $r\in\{i,\ldots,n\}$, a box has entry exactly $r$ if and only if it belongs to
\[
    \lambda_r\setminus\lambda_{r+1}.
\]
Note that if a box is in $\lambda_r$, its entry is at least $r$,
and if the box is not in $\lambda_{r+1}$, its entry is not at least $r+1$.

Therefore, every box of $\lambda_i$ has the property that a box has entry $r$ if and only if it is in $\lambda_r \setminus \lambda_{r+1}$. Hence, two tableaux satisfying the same growth conditions must agree in every box.
\end{proof}

\begin{theorem}\label{partial-tab-equal-rec}
    Using the definition of the recording tableau, 
    \[
        \rec(A_i)=T_i .
    \]
\end{theorem}

\begin{proof}
For each \(r\in\{i,i+1,\ldots,n\}\), let
\[
    I_r:=\operatorname{ins}(A_r)
    \qquad\text{and}\qquad
    I_r':=\operatorname{ins}(A_r').
\]
By Theorem~\ref{insertion-eq}, inserting the padded suffix matrix \(A_r\) gives the same insertion tableau as inserting the truncated suffix matrix \(A_r'\). Hence
\[
    I_r=I_r'
    \qquad\text{for every } r\geq i.
\]
In particular,
\[
    \operatorname{shape}(I_r)
    =
    \operatorname{shape}(I_r')
    =
    \operatorname{shape}\bigl(\operatorname{ins}(A_r')\bigr).
\]

Now consider the recording tableau \(\operatorname{rec}(A_i)\). By
definition, \(\operatorname{rec}(A_i)\) records the growth of the insertion tableau as the columns
\[
    C_n,C_{n-1},\ldots,C_i
\]
are inserted. Since the first \(i-1\) columns of \(A_i\) are zero columns, they create no boxes. Thus the only labels appearing in
\(\operatorname{rec}(A_i)\) are
\[
    i,i+1,\ldots,n.
\]

Let \(S_{\geq r}\) be the subtableau of
\(\operatorname{rec}(A_i)\) consisting of boxes with labels
\(r,r+1,\ldots,n\). Then, by the definition of the recording tableau,
\[
    \operatorname{shape}(S_{\geq r})
    =
    \operatorname{shape}(I_r)^{\mathsf t}
    =
    \operatorname{shape}\bigl(\operatorname{ins}(A_r')\bigr)^{\mathsf t}.
\]
Moreover, when we pass from \(A_r\) to \(A_{r-1}\), the only newly inserted column is \(C_{r-1}\). Therefore the new boxes in
\[
    S_{\geq r-1}\setminus S_{\geq r}
\]
are labeled \(r-1\).

Additionally, because \[
    \operatorname{shape}(S_{\geq r})
    =
    \operatorname{shape}\bigl(\operatorname{ins}(A_r')\bigr)^{\mathsf t}
    =
    \operatorname{shape}(T_r)
\] for every \(r\geq i\), the chain
\[
    S_{\geq n}\subseteq S_{\geq n-1}\subseteq\cdots\subseteq S_{\geq i}
\]
satisfies exactly the same shape and labeling conditions as the chain
\[
    T_n\subseteq T_{n-1}\subseteq\cdots\subseteq T_i
\]
used to define \(T_i\).

By Lemma~\ref{lem:growth-chain-uniqueness}, the tableau determined by this growth chain is unique. Therefore, $S_{\geq r}=T_r$ for every $r\ge i$, meaning
\[
    \operatorname{rec}(A_i)=T_i.
\]
\end{proof}

Let $R_i := A_i^\dagger$ or $R_i = \begin{bmatrix}
    \mathbf{1}_{(i-1) \times m} \\ C_i^\dagger \\ \vdots \\ C_n^\dagger
\end{bmatrix}$. 

Let $ E_{p\times q}
    :=
    \begin{bmatrix}
        \cdot & \cdots & \cdot\\
        \vdots & \ddots & \vdots\\
        \cdot & \cdots & \cdot
    \end{bmatrix}$
denote the \(p\times q\) empty array. Define $R_i'
    :=
    \begin{bmatrix}
        E_{(i-1)\times m}\\
        C_i^\dagger\\
        \vdots\\
        C_n^\dagger
    \end{bmatrix}.$

Thus \(R_i'\) is obtained from \(R_i\) by replacing its first
\(i-1\) rows of black checkers with empty rows. In particular,
\[
    R_1'=R_1=R.
\]

For a super pipe dream $P$ and an integer $i$, write
\[
    P_{\ge i}
\]
for the restriction of $P$ to rows $i,i+1,\ldots$.

\begin{theorem}[Lower-row invariance]
\label{thm:lower-row-invariance}
For every  \(2\leq i\leq n\), \[ 
\ins(R_i') \subseteq \ins(R_{i-1}') \]
and \[
\operatorname{number}(\ins(R_{i-1}') / \ins(R_i')) = i-1.
\]

\end{theorem}

\begin{proof}
The pipe dreams $R_i'$ and $R_{i-1}'$ agree in rows $i,i+1,\ldots,n$. Their only difference is that $R_{i-1}'$ has the additional row $C_{i-1}^{\dagger}$ in row $i-1$, while rows $1,\ldots,i-2$ are empty in both diagrams.

We compare their $Y^+$-rectification processes. Recall that red checkers move only to the right and weakly upward, while black checkers never change rows. Hence a move whose red checker begins above row $i$ cannot affect rows $i,\ldots,n$. Moreover, since $Y_j^+$ always acts on the lowest red checker in column $j$, adding checkers in row $i-1$ does not change the order in which red checkers lying in rows $i,\ldots,n$ are processed.

Now consider an elementary ladder move whose red checker lies in a row
$r\ge i$. Since the two diagrams agree in rows $i,\ldots,n$, there are
two possibilities.

Case 1: the ladder is not extended. The same ladder occurs in both diagrams. Therefore, exactly the same red and black checkers move in rows $i,\ldots,n$, so the two diagrams continue to agree in these rows.

Case 2: the ladder is extended. The only difference between the two diagrams lies strictly above row $i$. Consequently, the portion of the ladder lying in rows $i,\ldots,r$ is the same in both diagrams; only its upper end may extend farther in $R_{i-1}'$. Thus the ladder move performs exactly the same shifts on all checkers in rows $i,\ldots,n$, with any additional motion occurring strictly above row $i$.

It follows inductively over the elementary ladder moves that the rectified pipe dreams agree in rows $i,\ldots,n$. In particular, the black checkers in these rows undergo the same ladder moves and therefore have the same pipe connectivity in the two rectifications. Since $\operatorname{tab}$ records the row labels of these rectified black checkers, all entries labeled $i,\ldots,n$ occur in the same positions.

Hence,
\[
    \ins(R_i')
    \subseteq
    \ins(R_{i-1}').
\]
Finally, every additional black checker in $R_{i-1}'$ comes from the newly added row $i-1$, and black checkers never change rows during
rectification. Thus, every new tableau entry is labeled $i-1$. Therefore
\[
    \ins(R_{i-1}')
    \setminus
    \ins(R_i')
\]
consists entirely of entries labeled $i-1$.
\end{proof}

\begin{corollary}[Upper black rows do not change the lower insertion tableau]\label{cor:upper-black-row-invariance}
Let $\left.\operatorname{ins}(R_i)\right|_{\ge i}$ denote the subtableux of $R_i$ with entries strictly greater than $i$. 

\[
    \left.\operatorname{ins}(R_i)\right|_{\ge i}
    =
    \operatorname{ins}(R_i').
\]
\end{corollary}

\begin{proof}
The proof is the same lower-row invariance argument as
Theorem~\ref{thm:lower-row-invariance}. The additional black checkers occur strictly above row $i$, and black checkers never change rows under $Y^+$-rectification. Hence, the sequence of checker movements and the pipe connectivity in rows $i,i+1,\ldots$ are unchanged.
\end{proof}

\begin{lemma}[Rectangle bound for insertion tableaux]
\label{lem:rectangle-bound}
Let
\[
    A\in BM_{m\times n}.
\]
Then there exists a partition
\[
    \lambda\subseteq [m]\times[n]
\]
such that
\[
    \operatorname{ins}(A)\in \operatorname{RSSYT}(\lambda,m).
\]
In particular,
\[
    \operatorname{shape}(\operatorname{ins}(A))
    \subseteq [m]\times[n],
\]
so \(\operatorname{ins}(A)\) has at most \(m\) rows and at most \(n\)
columns.
\end{lemma}

\begin{proof}
First, the pipe-dream insertion tableau used here agrees with the ordinary column-insertion tableau. By \cite[Proposition~8.7 and the discussion preceding Proposition~8.8]{dennin2025cauchy},
if
\[
    \operatorname{Rect}(A)=(V,U),
\]
then
\[
    \operatorname{tab}(V)
    =
    \operatorname{ins}(A)
\]
where the right-hand side is the usual dual-RSK insertion tableau.

The usual dual RSK correspondence has codomain
\[
    \operatorname{dRSK}:
    BM_{m\times n}
    \longrightarrow
    \bigsqcup_{\lambda\subseteq[m]\times[n]}
    \operatorname{RSSYT}(\lambda,m)
    \times
    \operatorname{RSSYT}(\lambda^t,n),
\]
with
\[
    \operatorname{dRSK}(A)
    =
    \bigl(
        \operatorname{ins}(A),
        \operatorname{rec}(A)
    \bigr).
\]
Therefore, there exists a partition
\[
    \lambda\subseteq[m]\times[n]
\]
such that
\[
    \operatorname{ins}(A)
    \in
    \operatorname{RSSYT}(\lambda,m).
\]
Consequently
\[
    \operatorname{shape}(\operatorname{ins}(A))
    =
    \lambda
    \subseteq[m]\times[n].
\]
Since the rectangle $[m]\times[n]$ has $m$ rows and $n$ columns,
\[
    \ell(\lambda)\le m
    \qquad\text{and}\qquad
    \lambda_1\le n.
\]
Thus $\operatorname{ins}(A)$ has at most $m$ rows and at most $n$ columns.
\end{proof}

Let \(T\) be a tableau with entries in \(\{i,i+1,\ldots,n\}\) and with at most \(m\) columns. Define the filled tableau
\[
    \operatorname{fill}_{<i}(T)
\]
columnwise by
\[
    \operatorname{col}_{\operatorname{fill}_{<i}(T)}(q)
    :=
    [i-1]\cup\operatorname{col}_{T}(q)
\]
for every \(q\in[m]\), where
\[
    [i-1]:=\{1,2,\ldots,i-1\}.
\]
If \(T\) has fewer than \(m\) columns, the missing columns are treated as empty columns. Thus \(\operatorname{fill}_{<i}(T)\) is obtained by adjoining the entries \(1,\ldots,i-1\) to every one of the \(m\) possible columns.

When \(i\) is understood from context, we write simply
\[
    \operatorname{fill}(T)=\operatorname{fill}_{<i}(T).
\]

\begin{theorem}[Filling identity]
\label{filling-identity}
For every \(i\in[n]\),
\[
    \operatorname{fill}_{<i}\bigl(\operatorname{ins}(R_i')\bigr)
    =
    \operatorname{ins}(R_i).
\]
\end{theorem}

\begin{proof}
If $i=1$, then $R_1=R_1'$ and
$\operatorname{fill}_{<1}$ is the identity, so the result is immediate.

Now suppose $i\ge2$. The matrix $R_i$ is obtained from $R_i'$ by replacing the first $i-1$ empty rows with all-black rows. Since
\[
    R_i\in BM_{n\times m},
\]
each of these rows contains exactly $m$ black checkers.

By Corollary~\ref{cor:upper-black-row-invariance}, adding these rows does not change the part of the insertion tableau arising from rows
$i,i+1,\ldots,n$. Hence
\[
    \left.\operatorname{ins}(R_i)\right|_{\ge i}
    =
    \operatorname{ins}(R_i').
\]

It remains to determine the entries coming from the newly added all-black rows.

Black checkers never change rows under $Y^+$-rectification. Therefore,
for every $p<i$, the $m$ black checkers in row $p$ contribute exactly
$m$ entries labeled $p$ to $\ins(R_i)$.

By Lemma~\ref{lem:rectangle-bound}, the tableau $\ins(R_i)$ has at most $m$ columns. Since the columns of a reverse semistandard Young tableau are strictly decreasing, a fixed label $p$ can occur at most once in any column. Since it occurs exactly $m$ times and there are at most $m$ columns, it must occur exactly once in each of the $m$ possible columns.

Thus every column contains each of
\[
    1,2,\ldots,i-1
\]
exactly once.

All entries inherited from $\ins(R_i')$ are at least $i$. Because columns are strictly decreasing from top to bottom, the newly added smaller entries must occur below them, in the unique order
\[
    i-1,i-2,\ldots,1.
\]
Therefore, each column of $\ins(R_i)$ is obtained from the corresponding column of $\ins(R_i')$ by adjoining precisely the entries $1,\ldots,i-1$.

Hence,
\[
    \ins(R_i)
    =
    \operatorname{fill}_{<i}
    \bigl(
        \ins(R_i')
    \bigr).
\]
\end{proof} 

Let $T$ be a reverse semistandard Young tableau with entries in $[r]$ and with at most $c$ columns. We regard $T$ as lying inside the full
$r\times c$ rectangle, treating missing columns as empty.

Recall that the complementary tableau $\overline{T}$ is defined by
\[
    \operatorname{col}_q(\overline{T})
    :=
    [r]\setminus
    \operatorname{col}_{c+1-q}(T),
    \qquad q\in[c].
\]

\begin{definition}[Complementary positions]
\label{def:complementary-positions}
Let $S$ and $T$ be tableaux with entries in $[r]$ and at most $c$ columns.

For $a\in[r]$, we say that the entries labeled $a$ occur in \emph{complementary positions} in $S$ and $T$ if, for every $q\in[c]$,
\[
    a\in\operatorname{col}_q(S)
    \quad\Longleftrightarrow\quad
    a\notin\operatorname{col}_{c+1-q}(T).
\]

We say that $S$ and $T$ are \emph{complementary} if every $a\in[r]$ occurs in complementary positions.
\end{definition}

By definition,
\[
    S=\overline{T}
\]
if and only if $S$ and $T$ are complementary.

\begin{proposition}[Shape compatibility]
\label{prop:shape-compatibility}
For every \(A\in \mathrm{BM}_{m\times n}\),
\[
    \operatorname{shape}\bigl(\operatorname{ins}(A^\dagger)\bigr)
    =
    \operatorname{shape}\bigl(\overline{\operatorname{rec}(A)}\bigr).
\]
In particular, for every \(i\in[n]\),
\[
    \operatorname{shape}\bigl(\operatorname{ins}(R_i)\bigr)
    =
    \operatorname{shape}\bigl(\overline{\operatorname{rec}(A_i)}\bigr).
\]
\end{proposition}

\begin{proof}
Let
\[
    \lambda:=\operatorname{shape}\bigl(\operatorname{ins}(A)\bigr).
\]
By Proposition~8.8 of \cite{dennin2025cauchy}, the pipe-dream version of
dual RSK satisfies
\[
    \operatorname{dRSK}'(A)
    =
    \bigl(\operatorname{ins}(A),\operatorname{ins}(A^\dagger)\bigr)
    \in
    \operatorname{RSSYT}(\lambda,m)
    \times
    \operatorname{RSSYT}(\lambda^\dagger,n).
\]
It follows that
\[
    \operatorname{shape}\bigl(\operatorname{ins}(A^\dagger)\bigr)
    =
    \lambda^\dagger.
\]

On the other hand, the usual dual RSK correspondence gives
\[
    \operatorname{dRSK}(A)
    =
    \bigl(\operatorname{ins}(A),\operatorname{rec}(A)\bigr)
    \in
    \operatorname{RSSYT}(\lambda,m)
    \times
    \operatorname{RSSYT}(\lambda^{\mathsf t},n).
\]
Hence
\[
    \operatorname{shape}\bigl(\operatorname{rec}(A)\bigr)
    =
    \lambda^{\mathsf t}.
\]

By Lemma \ref{lem:rectangle-bound}, $\lambda\subseteq [m]\times[n]$, meaning we can write
\[
    \lambda=(\lambda_1,\lambda_2,\ldots,\lambda_m),
\] where $\lambda_j=0$ for any missing rows. The $q$-th column of
$\lambda^{\mathsf t}$ has height $\lambda_q$. Therefore,
\[
    \left|\operatorname{col}_q(\rec(A))\right|=\lambda_q
    \qquad\text{for every }q\in[m].
\]

Recall that the complementary tableau is defined columnwise by
\[
    \operatorname{col}_q(\overline{\rec(A)})
    :=
    [n]\setminus
    \operatorname{col}_{m+1-q}(\rec(A)).
\]
Hence,
\begin{align*}
    \left|\operatorname{col}_q(\overline{\rec(A)})\right|
    &=
    n-
    \left|\operatorname{col}_{m+1-q}(\rec(A))\right|\\
    &=
    n-\lambda_{m+1-q}.
\end{align*}

By definition, $\lambda^\dagger$ is the partition in the $n\times m$
rectangle whose $q$-th column has height
\[
    n-\lambda_{m+1-q},
    \qquad q\in[m].
\]
Thus, $\overline{\rec(A)}$ and $\lambda^\dagger$ have the same column heights in every column, and therefore
\[
    \operatorname{shape}(\overline{\rec(A)})
    =
    \lambda^\dagger.
\]

Combining the two equalities yields
\[
    \operatorname{shape}\bigl(\ins(A^\dagger)\bigr)
    =
    \operatorname{shape}\bigl(\overline{\rec(A)}\bigr).
\]

Applying this result to \(A_i\), and using \(R_i=A_i^\dagger\), gives
\[
    \operatorname{shape}\bigl(\ins(R_i)\bigr)
    =
    \operatorname{shape}\bigl(\overline{\rec(A_i)}\bigr).
\]
\end{proof}

\begin{theorem}
\label{prop:base-case}
Let \(A_n\) be the suffix matrix retaining only the final column of
\(A\), and let \(R_n=A_n^\dagger\). Then
\[
    \ins(R_n)
    =
    \overline{\rec(A_n)}.
\]
\end{theorem}

\begin{proof}
Let
\[
    C_n
    =
    \{r\in[m]:A_{rn}=1\},
    \qquad
    k:=|C_n|.
\]
Since $A_n$ has only one nonzero column, its $k$ inserted boxes are all created by column $n$. Therefore, by reverse semi-standard notation of tableaux, 
\[
    \rec(A_n)
    =
    \underbrace{
        \begin{array}{|c|c|c|c|}
        \hline
        n & n & \cdots & n\\
        \hline
        \end{array}
    }_{k\text{ boxes}}.
\]

Now consider $\ins(R_n)$.

By Theorem~\ref{filling-identity},
\[
    \ins(R_n)
    =
    \operatorname{fill}_{<n}
    \bigl(
        \ins(R_n')
    \bigr).
\]
Consequently, every one of the $m$ possible columns of $\ins(R_n)$ contains each of the entries
\[
    1,2,\ldots,n-1
\]
exactly once. Since columns are strictly decreasing, these entries occur in the forced order
\[
    n-1,n-2,\ldots,1.
\]
Any additional box in a column must contain the only remaining possible label, $n$.

By Proposition~\ref{prop:shape-compatibility},
\[
    \operatorname{shape}
    \bigl(
        \ins(R_n)
    \bigr)
    =
    \operatorname{shape}
    \bigl(
        \overline{\rec(A_n)}
    \bigr).
\]
By definition, the complementary tableau $\overline{\rec(A_n)}$ contains $1,\ldots,n-1$ once in every column because there are no $1,\ldots,n-1$ in $\rec(A_n)$, and its remaining boxes are filled
with $n$.

The two tableaux have the same shape, their entries $1,\ldots,n-1$ are forced into the same positions, and every remaining box in each tableau is forced to contain $n$. Hence,
\[
    \ins(R_n)
    =
    \overline{\rec(A_n)}.
\]
\end{proof} 

\begin{theorem} \label{thm:induction-step}
    If $\overline{\rec(A_i)} = \ins(R_i)$ for $i \in [n]$ and $i > 1$, then $\overline{\rec(A_{i-1})} = \ins(R_{i-1})$.
\end{theorem}

\begin{proof}
    Let
    \[
        S_i:=\ins(R_i),
        \qquad
        T_i:=\rec(A_i),
    \]
    and
    \[
        S_{i-1}:=\ins(R_{i-1}),
        \qquad
        T_{i-1}:=\rec(A_{i-1}).
    \]
    
    Because $S_i=\overline{T_i},$ every label in $[n]$ occurs in complementary positions in $S_i$ and $T_i$. We want to show the same is true for $S_{i-1}$ and $T_{i-1}$.

    \medskip

    By Theorem~4.6,
    \[
        T_i\subseteq T_{i-1},
    \]
    and transforming from $T_i$ to $T_{i-1}$ adds only boxes labeled $i-1$. Therefore, the positions of every label $i,i+1,\ldots,n$
    are unchanged.
    
    On the insertion side, Theorem~\ref{thm:lower-row-invariance} gives
    \[
        \ins(R_i')
        \subseteq
        \ins(R_{i-1}'),
    \]
    with only entries $i-1$ added. By Theorem \ref{filling-identity}, $\operatorname{fill}(R_i)$ and $\operatorname{fill}(R_{i-1})$ adds only labels below $i$ and therefore does not change any entry labeled $\geq i$.
    
    Hence, the positions of
    \[
        i,i+1,\ldots,n
    \]
    in $S_{i-1}$ are the same as their positions in $S_i$. Since these labels occur in complementary positions in $S_i$ and $T_i$ by the induction hypothesis, they also occur in complementary positions in $S_{i-1}$ and $T_{i-1}$.

    \medskip

    By Theorem \ref{filling-identity},
    \[
        S_{i-1}
        =
        \operatorname{fill}_{<i-1}
        \bigl(
            \ins(R_{i-1}')
        \bigr).
    \]
    Meaning, every label
    \[
        1,2,\ldots,i-2
    \]
    occurs once in every one of the $m$ possible columns of $S_{i-1}$.

    On the other hand, $A_{i-1}$ has no black checkers in columns $1,\ldots,i-2$. Therefore, its recording tableau $T_{i-1}$ has no label in $[i-2]$.

    Therefore, every label
    \[
        1,2,\ldots,i-2
    \]
    occurs in complementary positions in $S_{i-1}$ and $T_{i-1}$.
    
    \medskip
    
    Now, it only remains to prove that the label $i-1$ are in complementary positions.

    Let \[
    \operatorname{col}_{q}(S_{i-1})
    \] denote the $q$-th column of the tableaux $S_{i-1}$ and \[
    |\operatorname{col}_{q}(S_{i-1})|
    \] denote the number of boxes in the $q$-th column of tableaux $S_{i-1}$.

    By Proposition \ref{prop:shape-compatibility},
    \[
        \operatorname{shape}(S_{i-1})
        =
        \operatorname{shape}(\overline{T_{i-1}}).
    \] Hence, for each column $q$ in $S_{i-1}$ and $m-q+1$ in $T_{i-1}$ their corresponding column heights are complementary or \[|\operatorname{col}_{q}(S_{i-1})| + |\operatorname{col}_{m-q+1}(T_{i-1})| = n.\] 
    
    Because we know that entries $[n, i]$ and $[i-2, 0]$ are already in complementary positions, for each $\operatorname{col}_{q}(S_{i-1})$ and $\operatorname{col}_{m-q+1}(T_{i-1})$, $n-1$ positions are already taken. 
    
    Therefore, for each column $q$, the entry $i-1$ can only belong to either $\operatorname{col}_{q}(S_{i-1})$ or $\operatorname{col}_{m-q+1}(T_{i-1})$ but not both because $|\operatorname{col}_{q}(S_{i-1})| + |\operatorname{col}_{m-q+1}(T_{i-1})| = n$. 
    
    Hence, the entries labeled $i-1$ occur in complementary positions, meaning every label in $[n]$ occurs in complementary positions in $S_{i-1}$ and $T_{i-1}$. 
    
    By definition \ref{def:complementary-positions},
    \[
        S_{i-1}
        =
        \overline{T_{i-1}}.
    \] meaning, 
    \[
        \ins(R_{i-1})
        =
        \overline{\rec(A_{i-1})}.
    \]
\end{proof}

Now, we can move to prove $\overline{\rec(A)} = \ins(A^\dagger)$.

\begin{proof}[Proof of Theorem \ref{thm:main-thm}]
    
    By Theorem \ref{prop:base-case}, for any $A \in BM_{m \times n}$, $\ins(R_n) = \overline{\rec(A_n)}$. Additionally, by Theorem \ref{thm:induction-step}, if $\overline{\rec(A_i)} = \ins(R_i)$ for $i \in [n]$ and $i > 1$, then $\overline{\rec(A_{i-1})} = \ins(R_{i-1})$. Therefore, by induction, $\overline{\rec(A_{1})} = \ins(R_{1})$. 
    
    Furthermore, because $A_1 = A$ and $R_1 = R = A^\dagger$, $\overline{\rec(A)} = \ins(A^\dagger)$. 
    
\end{proof}

\bibliography{bibliography}

@article{dennin2025cauchy,
  title={Cauchy identities for Grothendieck polynomials and a dual RSK correspondence through pipe dreams},
  author={Dennin, Hugh},
  journal={arXiv preprint arXiv:2506.21052},
  year={2025}
}

@article {knuth1970permutation,
    AUTHOR = {Knuth, Donald E.},
     TITLE = {Permutations, matrices, and generalized {Y}oung tableaux},
   JOURNAL = {Pacific J. Math.},
  FJOURNAL = {Pacific Journal of Mathematics},
    VOLUME = {34},
      YEAR = {1970},
     PAGES = {709--727},
      ISSN = {0030-8730,1945-5844},
   MRCLASS = {05.30},
  MRNUMBER = {272654},
MRREVIEWER = {M.\ Doob},
       URL = {http://projecteuclid.org/euclid.pjm/1102971948},
}

@article {robinson1938on,
    AUTHOR = {Robinson, G. de B.},
     TITLE = {On the {R}epresentations of the {S}ymmetric {G}roup},
   JOURNAL = {Amer. J. Math.},
  FJOURNAL = {American Journal of Mathematics},
    VOLUME = {60},
      YEAR = {1938},
    NUMBER = {3},
     PAGES = {745--760},
      ISSN = {0002-9327,1080-6377},
   MRCLASS = {99-04},
  MRNUMBER = {1507943},
       DOI = {10.2307/2371609},
       URL = {https://doi.org/10.2307/2371609},
}

@article {schensted1961longest,
    AUTHOR = {Schensted, C.},
     TITLE = {Longest increasing and decreasing subsequences},
   JOURNAL = {Canadian J. Math.},
  FJOURNAL = {Canadian Journal of Mathematics. Journal Canadien de
              Math\'ematiques},
    VOLUME = {13},
      YEAR = {1961},
     PAGES = {179--191},
      ISSN = {0008-414X,1496-4279},
   MRCLASS = {05.00},
  MRNUMBER = {121305},
MRREVIEWER = {D.\ E.\ Rutherford},
       DOI = {10.4153/CJM-1961-015-3},
       URL = {https://doi.org/10.4153/CJM-1961-015-3},
}

@book {stanley2024enumerative,
    AUTHOR = {Stanley, Richard P.},
     TITLE = {Enumerative combinatorics. {V}ol. 2},
    SERIES = {Cambridge Studies in Advanced Mathematics},
    VOLUME = {208},
   EDITION = {Second},
      NOTE = {With an appendix by Sergey Fomin},
 PUBLISHER = {Cambridge University Press, Cambridge},
      YEAR = {[2024] \copyright 2024},
     PAGES = {xvi+783},
      ISBN = {978-1-009-26249-1; 978-1-009-26248-4},
   MRCLASS = {05-02 (05A15 05E05 05E10 68R05)},
  MRNUMBER = {4621625},
MRREVIEWER = {Timothy\ Y.\ Chow},
}

@article {bufetov2018hall,
    AUTHOR = {Bufetov, Alexey and Matveev, Konstantin},
     TITLE = {Hall-{L}ittlewood {RSK} field},
   JOURNAL = {Selecta Math. (N.S.)},
  FJOURNAL = {Selecta Mathematica. New Series},
    VOLUME = {24},
      YEAR = {2018},
    NUMBER = {5},
     PAGES = {4839--4884},
      ISSN = {1022-1824,1420-9020},
   MRCLASS = {05E05 (60K35)},
  MRNUMBER = {3874706},
MRREVIEWER = {Maciej\ Do\l\polhk ega},
       DOI = {10.1007/s00029-018-0442-y},
       URL = {https://doi.org/10.1007/s00029-018-0442-y},
}

@article {frieden2024qt,
    AUTHOR = {Frieden, Gabriel and Schreier-Aigner, Florian},
     TITLE = {{$qt{\rm RSK}^*$}: a probabilistic dual {RSK} correspondence
              for {M}acdonald polynomials},
   JOURNAL = {S\'em. Lothar. Combin.},
  FJOURNAL = {S\'eminaire Lotharingien de Combinatoire},
    VOLUME = {91B},
      YEAR = {2024},
     PAGES = {Art. 75, 12},
      ISSN = {1286-4889},
   MRCLASS = {05E05 (33C52 60C05)},
  MRNUMBER = {4818707},
}

@article {sagan1990robinson,
    AUTHOR = {Sagan, Bruce E. and Stanley, Richard P.},
     TITLE = {Robinson-{S}chensted algorithms for skew tableaux},
   JOURNAL = {J. Combin. Theory Ser. A},
  FJOURNAL = {Journal of Combinatorial Theory. Series A},
    VOLUME = {55},
      YEAR = {1990},
    NUMBER = {2},
     PAGES = {161--193},
      ISSN = {0097-3165,1096-0899},
   MRCLASS = {05E05 (05E10)},
  MRNUMBER = {1075706},
MRREVIEWER = {Dennis\ White},
       DOI = {10.1016/0097-3165(90)90066-6},
       URL = {https://doi.org/10.1016/0097-3165(90)90066-6},
}

@article {corwin2014tropical,
    AUTHOR = {Corwin, Ivan and O'Connell, Neil and Sepp\"al\"ainen, Timo and
              Zygouras, Nikolaos},
     TITLE = {Tropical combinatorics and {W}hittaker functions},
   JOURNAL = {Duke Math. J.},
  FJOURNAL = {Duke Mathematical Journal},
    VOLUME = {163},
      YEAR = {2014},
    NUMBER = {3},
     PAGES = {513--563},
      ISSN = {0012-7094,1547-7398},
   MRCLASS = {05E05 (05E10 60B20 60K35 82B23)},
  MRNUMBER = {3165422},
MRREVIEWER = {Marius\ R\u adulescu},
       DOI = {10.1215/00127094-2410289},
       URL = {https://doi.org/10.1215/00127094-2410289},
}

@article {aigner2022qrst,
    AUTHOR = {Aigner, Florian and Frieden, Gabriel},
     TITLE = {{$q{\rm RS}t$}: a probabilistic {R}obinson-{S}chensted
              correspondence for {M}acdonald polynomials},
   JOURNAL = {Int. Math. Res. Not. IMRN},
  FJOURNAL = {International Mathematics Research Notices. IMRN},
      YEAR = {2022},
    NUMBER = {17},
     PAGES = {13505--13568},
      ISSN = {1073-7928,1687-0247},
   MRCLASS = {60C05 (05E05 33C52)},
  MRNUMBER = {4475282},
       DOI = {10.1093/imrn/rnab083},
       URL = {https://doi.org/10.1093/imrn/rnab083},
}

@incollection {fomingrothendieck,
    AUTHOR = {Fomin, Sergey and Kirillov, Anatol N.},
     TITLE = {Grothendieck polynomials and the {Y}ang-{B}axter equation},
 BOOKTITLE = {Formal power series and algebraic combinatorics/{S}\'eries
              formelles et combinatoire alg\'ebrique},
     PAGES = {183--189},
 PUBLISHER = {DIMACS, Piscataway, NJ},
      YEAR = {sd},
   MRCLASS = {05E05 (14N15)},
  MRNUMBER = {2307216},
}

@inproceedings {fomin1996the,
    AUTHOR = {Fomin, Sergey and Kirillov, Anatol N.},
     TITLE = {The {Y}ang-{B}axter equation, symmetric functions, and
              {S}chubert polynomials},
 BOOKTITLE = {Proceedings of the 5th {C}onference on {F}ormal {P}ower
              {S}eries and {A}lgebraic {C}ombinatorics ({F}lorence, 1993)},
   JOURNAL = {Discrete Math.},
  FJOURNAL = {Discrete Mathematics},
    VOLUME = {153},
      YEAR = {1996},
    NUMBER = {1-3},
     PAGES = {123--143},
      ISSN = {0012-365X,1872-681X},
   MRCLASS = {05E05 (14C17 14M15 82B23)},
  MRNUMBER = {1394950},
MRREVIEWER = {Jean-Yves\ Thibon},
       DOI = {10.1016/0012-365X(95)00132-G},
       URL = {https://doi.org/10.1016/0012-365X(95)00132-G},
}

@article {bergeron1993rc,
    AUTHOR = {Bergeron, Nantel and Billey, Sara},
     TITLE = {R{C}-graphs and {S}chubert polynomials},
   JOURNAL = {Experiment. Math.},
  FJOURNAL = {Experimental Mathematics},
    VOLUME = {2},
      YEAR = {1993},
    NUMBER = {4},
     PAGES = {257--269},
      ISSN = {1058-6458,1944-950X},
   MRCLASS = {05E99 (05E05 14M15 20C30)},
  MRNUMBER = {1281474},
MRREVIEWER = {Axel\ Kohnert},
       URL = {http://projecteuclid.org/euclid.em/1048516036},
}

@article {knutson2005grobner,
    AUTHOR = {Knutson, Allen and Miller, Ezra},
     TITLE = {Gr\"obner geometry of {S}chubert polynomials},
   JOURNAL = {Ann. of Math. (2)},
  FJOURNAL = {Annals of Mathematics. Second Series},
    VOLUME = {161},
      YEAR = {2005},
    NUMBER = {3},
     PAGES = {1245--1318},
      ISSN = {0003-486X,1939-8980},
   MRCLASS = {05E15 (13C40 13F55 13P10 14M15 14N15)},
  MRNUMBER = {2180402},
MRREVIEWER = {Harry\ Tamvakis},
       DOI = {10.4007/annals.2005.161.1245},
       URL = {https://doi.org/10.4007/annals.2005.161.1245},
}

@article {lascoux1982polynomes,
    AUTHOR = {Lascoux, Alain and Sch\"utzenberger, Marcel-Paul},
     TITLE = {Polyn\^omes de {S}chubert},
   JOURNAL = {C. R. Acad. Sci. Paris S\'er. I Math.},
  FJOURNAL = {Comptes Rendus des S\'eances de l'Acad\'emie des Sciences.
              S\'erie I. Math\'ematique},
    VOLUME = {294},
      YEAR = {1982},
    NUMBER = {13},
     PAGES = {447--450},
      ISSN = {0249-6291},
   MRCLASS = {14M17 (05A10 14N10)},
  MRNUMBER = {660739},
}

@article {lascoux1982structure,
    AUTHOR = {Lascoux, Alain and Sch\"utzenberger, Marcel-Paul},
     TITLE = {Structure de {H}opf de l'anneau de cohomologie et de l'anneau
              de {G}rothendieck d'une vari\'et\'e{} de drapeaux},
   JOURNAL = {C. R. Acad. Sci. Paris S\'er. I Math.},
  FJOURNAL = {Comptes Rendus des S\'eances de l'Acad\'emie des Sciences.
              S\'erie I. Math\'ematique},
    VOLUME = {295},
      YEAR = {1982},
    NUMBER = {11},
     PAGES = {629--633},
      ISSN = {0249-6291},
   MRCLASS = {14M17},
  MRNUMBER = {686357},
}

@incollection {lascoux1983symmetry,
    AUTHOR = {Lascoux, Alain and Sch\"utzenberger, Marcel-Paul},
     TITLE = {Symmetry and flag manifolds},
 BOOKTITLE = {Invariant theory ({M}ontecatini, 1982)},
    SERIES = {Lecture Notes in Math.},
    VOLUME = {996},
     PAGES = {118--144},
 PUBLISHER = {Springer, Berlin},
      YEAR = {1983},
      ISBN = {3-540-12319-9},
   MRCLASS = {14M17 (20B30 32M10)},
  MRNUMBER = {718129},
MRREVIEWER = {Konrad\ Drechsler},
       DOI = {10.1007/BFb0063238},
       URL = {https://doi.org/10.1007/BFb0063238},
}

@article {lascoux1982classes,
    AUTHOR = {Lascoux, Alain},
     TITLE = {Classes de {C}hern des vari\'et\'es de drapeaux},
   JOURNAL = {C. R. Acad. Sci. Paris S\'er. I Math.},
  FJOURNAL = {Comptes Rendus des S\'eances de l'Acad\'emie des Sciences.
              S\'erie I. Math\'ematique},
    VOLUME = {295},
      YEAR = {1982},
    NUMBER = {5},
     PAGES = {393--398},
      ISSN = {0249-6291},
   MRCLASS = {14M17 (32M10 55R40)},
  MRNUMBER = {684734},
MRREVIEWER = {Konrad\ Drechsler},
}
\bibliographystyle{alpha}

\end{document}